\documentclass[11pt]{article}
\usepackage[numbers,sort&compress]{natbib}
\usepackage{enumerate}
\usepackage{amscd}
\usepackage{array}
\usepackage{amsmath}
\usepackage{latexsym}
\usepackage{amsfonts}
\usepackage{amssymb}
\usepackage{amsthm}
\usepackage{verbatim}
\usepackage{mathrsfs}
\usepackage{enumerate}
\usepackage{hyperref}

\theoremstyle{plain}
\theoremstyle{definition}\newtheorem{theorem}{Theorem}[section]
\theoremstyle{plain}\newtheorem{lemma}[theorem]{Lemma}
\theoremstyle{plain}\newtheorem{coro}[theorem]{Corollary}
\theoremstyle{plain}\newtheorem{prop}[theorem]{Proposition}
\theoremstyle{remark}\newtheorem{remark}{Remark}[section]
\theoremstyle{definition}
\theoremstyle{plain}
\usepackage{xcolor}
\newcommand{\wred}[1]{\textcolor{black}{#1}}
\newcommand{\hs}[1]{\textcolor{black}{#1}}

\newcommand{\wg}[1]{\textcolor{black}{#1}}

\newcommand{\Div}{\mathrm{div}\,}
\newcommand{\B}{\Big}

\newcommand{\be}{\begin{equation}}
\newcommand{\ee}{\end{equation}}
 \newcommand{\ba}{\begin{aligned}}
 \newcommand{\ea}{\end{aligned}}

  \newcommand{\f}{\frac}
    
  \newcommand{\ben}{\begin{enumerate}}
   \newcommand{\een}{\end{enumerate}}

\newcommand{\ti}{\nabla}

\newcommand{\Rmnum}[1]{\expandafter\@slowromancap\romannumeral #1@}

\allowdisplaybreaks

\numberwithin{equation}{section}
\begin{document}
\title{On continuation criteria for the full compressible Navier-Stokes equations in Lorentz spaces
  }
	\author{Yanqing Wang\footnote{ Department of Mathematics and Information Science, Zhengzhou University of Light Industry, Zhengzhou, Henan  450002,  P. R. China Email: wangyanqing20056@gmail.com}, ~\, Wei Wei\footnote{Center for Nonlinear Studies, School of Mathematics, Northwest University, Xi'an, Shaanxi 710127, P. R. China Email: ww5998198@126.com }~~,\;~Gang Wu\footnote{School of Mathematical Sciences,  University of Chinese Academy of Sciences, Beijing 100049, P. R. China Email: wugang2011@ucas.ac.cn}~~~~and ~\,
Yulin Ye\footnote{School of Mathematics and Statistics,
Henan University,
Kaifeng, 475004,
P. R. China Email: ylye@vip.henu.edu.cn   }}
\date{}
\maketitle
\begin{abstract}
 In this paper, we derive several new sufficient conditions of non-breakdown of strong solutions for both the 3D heat-conducting compressible Navier-Stokes system and nonhomogeneous
incompressible Navier-Stokes equations. First, it is shown that  there exists a positive constant $\varepsilon$ such that the solution $(\rho,u,\theta)$ to full compressible
Navier-Stokes equations can be extended   beyond $t=T$ provided that   one of the following two conditions holds
  \begin{enumerate}[(1)]
 \item  $\rho \in L^{\infty}(0,T;L^{\infty}(\mathbb{R}^{3}))$,   $u\in L^{p,\infty}(0,T;L^{q,\infty}(\mathbb{R}^{3}))$ and
     \begin{equation}\label{L1}\| u\|_{L^{p,\infty}(0,T;L^{q,\infty}(\mathbb{R}^{3}))}\leq \varepsilon, ~~\text{with}~~ {2/p}+ {3/q}=1,\ \  q>3;\end{equation}
      \item $\lambda<3\mu,$ $\rho \in L^{\infty}(0,T;L^{\infty}(\mathbb{R}^{3}))$, $\theta\in L^{p,\infty}(0,T;L^{q,\infty}(\mathbb{R}^{3}))$ and
  \begin{equation}\label{L12}\|\theta\|_{L^{p,\infty}(0,T;
    L^{q,\infty}(\mathbb{R}^{3}))}\leq \varepsilon, ~~\text{with}~~ {2/p}+ {3/q}=2,\ \  q>3/2.\end{equation}
  \end{enumerate}
  To the best of our knowledge, this is the first continuation theorem allowing the time
direction to  be in Lorentz spaces for the   compressible
fluid. Second,   we establish some blow-up criteria in anisotropic Lebesgue spaces, for the finite blow-up time $T^{\ast}$,
 \begin{enumerate}[(1)]
 \item   Assume that the pair $(p,\overrightarrow{q})$ satisfies $ {2}/{p }+{1}/{q_{1} }+{1}/{q_{2} }+{1}/{q_{3} }=1$ $(1<q_{i}<\infty)$ and  \eqref{xianzhifanwei}, then
 \begin{equation}\label{AL1}\limsup_{t\rightarrow T^*}\B(
  \|\rho \|_{L^{\infty}(0,t;L^{\infty}(\mathbb{R}^{3}))}+ \|  u \|_{L^{p }(0,t; L_{1}^{ q_{1} }L_{2}^{ q_{2} }
L_{3}^{q_{3} }(\mathbb{R}^{3}) )} \B)= \infty,
  \end{equation}
   \item Let the pair $(p,\overrightarrow{q})$ satisfy ${2}/{p }+{1}/{q_{1} }+{1}/{q_{2} }+{1}/{q_{3} }=2$ $(1<q_{i}<\infty)$ and  \eqref{xianzhifanwei}, then
    \begin{equation}\label{AL2}\limsup_{t\rightarrow T^*}\B(
  \|\rho \|_{L^{\infty}(0,t;L^{\infty}(\mathbb{R}^{3}))}+ \| \theta \|_{L^{p }(0,t; L_{1}^{ q_{1} }L_{2}^{ q_{2} }
L_{3}^{q_{3} }(\mathbb{R}^{3}) )} \B)= \infty, (\lambda<3\mu).
  \end{equation}
   \end{enumerate}
   Third,
 without the condition on $\rho$ in \eqref{L1} and \eqref{AL1},  the results also hold  for the 3D nonhomogeneous incompressible Navier-Stokes equations. The appearance of vacuum in these systems could be allowed.
 \end{abstract}
\noindent {\bf MSC(2000):}\quad 76D03, 76D05, 35B33, 35Q35 \\\noindent
{\bf Keywords:} Navier-Stokes equations;   weak solutions;   regularity \\
\section{Introduction}
\label{intro}
\setcounter{section}{1}\setcounter{equation}{0}
In this paper, we are concerned with the
 following 3D full compressible Navier-Stokes  equations
\be\left\{\ba\label{FNS}
	&\rho_t+\nabla \cdot (\rho u)=0, \\
&\rho  u_t+\rho u\cdot\nabla u+\nabla
P(\rho,\theta)-\mu\Delta u-(\mu+\lambda)\nabla\text{div\,}u=0,\\
& c_{v}[\rho \theta_t+\rho u\cdot\nabla\theta] +P\text{div\,}u-\kappa\Delta\theta=\frac{\mu}{2}\left|\nabla u+(\nabla u)^{\text{tr}}\right|^2+\lambda(\text{div\,}u)^2,\\&(\rho, u, \theta)|_{t=0}=(\rho_0, u_0, \theta_0),
	\ea\right.\ee
 where $\rho,\, u,\, \theta$ stand for the flow density, velocity and the absolute temperature, respectively. The scalar function $P$ represents the   pressure, the state equation of which is  determined by
 \be
 P=R\rho\theta,  R>0,
 \ee
 and $\kappa$ is a positive constant.   $\mu$ and $\lambda$ are the coefficients of viscosity, which are assumed to be constants, satisfying the following physical restrictions:
\be\label{nares}
\mu>0,\ 2\mu+3\lambda\ge0.\ee  The
initial
conditions satisfy
\be\label{non-boundary}
\rho(x,t)\rightarrow0,\ u(x,t)\rightarrow0,\ \theta(x,t)\rightarrow0,\ \mathrm{as}\ |x|\rightarrow\infty,\ \mathrm{for}\ t\ge0.
\ee
It is clear that if  the triplet $(\rho(x,t),u(x, t),\theta(x, t) )  $ \hs{solve} system \eqref{FNS}, then the triplet $(\rho_{\lambda},u_{\lambda},\theta_{\lambda})  $ is also a solution of \eqref{FNS} for any $\lambda\in \hs{\mathbb{R}^{+}},$ where
\be\label{eqscaling}
\rho_{\lambda}=  \rho\hs{(\lambda x,\lambda^{2}t)},~~~~~u_{\lambda}=\lambda u\hs{(\lambda x,\lambda^{2}t)},~~~~~\theta_{\lambda}=\lambda^{2} \theta\hs{(\lambda x,\lambda^{2}t)}.
\ee
The scaling of velocity $u$ is the same as the incompressible Navier-Stokes system
\be\left\{\ba\label{NS}
&u_{t} -\Delta  u+ u\cdot\ti
u  +\nabla \Pi=0, \\
&\Div u=0,\\
&u|_{t=0}=u_0,
\ea\right.\ee
 In contrast with the pressure $P=R\rho\theta$ in \eqref{FNS}, the \hs{pressure} $\Pi$ in \eqref{NS} is determined by $\Delta\Pi=-$divdiv$(u\otimes u)$.
 The global \hs{well-posedness} of the 3D Navier-Stokes equations \eqref{NS} is an outstanding problem.
 It is well known that the \hs{solution $u$ to}
 \eqref{NS} is regular on $(0,T]$ if $u$ \hs{satisfy}
  \be \label{serrin1}
  u\in  L^{p} (0,T;L^{q}( \mathbb{R}^{3})) ~~~ \text{with}~~~~2/p+3/q=1, ~~q>3.
  \ee
   This is so-called Serrin  type regularity criteria \hs{for} the incompressible Navier-Stokes system obtained in \cite{[Serrin],[Struwe]}. Since
 Lorentz spaces $L^{r,s}(\mathbb{R}^{3})$ $(s\geq r)$ are larger than the \hs{Lebesgue}
spaces $L^{r }(\mathbb{R}^{3})$ and enjoy the same scaling as
  $L^{r}(\mathbb{R}^{3})$,
  \hs{Chen-Price} \cite{[CP]},  Sohr \cite{[Sohr]}, and
 \hs{Kozono-Kim}   \cite{[KK]}   presented an improvement
of \eqref{serrin1} by
  \be \label{serrin1L}
  u\in  L^{p,\infty} (0,T;L^{q,\infty}( \mathbb{R}^{3})) ~\text{and}~  \|u\|_{L^{p,\infty} (0,T;L^{q,\infty}( \mathbb{R}^{3}))}\leq\varepsilon~ \text{with}~2/p+3/q=1, ~~q>3.
  \ee
Alternative \hs{proofs} of \eqref{serrin1L} are given in \cite{[BPR],[WWY],[GKT]}.

We turn our attention back to the  full compressible fluid \eqref{FNS}. The classical theory of  compressible flow can be found in \cite{[Lions],[Feireisl]}. Here,
we are concerned with the regularity of local
strong solutions \hs{to} equations \eqref{FNS} established by Cho and   Kim in \cite{[CKC]} (\hs{For} details, see Proposition \ref{localwith vacuum} in \hs{Section} 2).
This \hs{type of} strong \hs{solution allows} the initial data to be in vacuum.
 In the spirit of \eqref{serrin1}, many authors successfully extended Serrin type \hs{blow-up}  criteria \eqref{serrin1} to the  compressible flow (\hs{e.g.} \cite{[HLX],[HLW],[JWY],[WZ13],[CY],[LX],[SWZ1],[SWZ2],[XZ],[LX]}  and references therein).
In particular, Huang, Li and Wang \cite{[HLW]} proved the following necessary condition for \hs{blow-up of solutions} to \eqref{FNS}
\be \label{HLW}
\lim\sup\limits_{\hs{t\rightarrow
T^\star}}\left(\|\text{div\,}u\|_{L^1(0,\hs{t};\hs{L^\infty}(\mathbb{R}^{3}))}+\|u\|_{L^p(0,\hs{t};L^q(\mathbb{R}^{3}))}\right)=\infty, ~~\f2p+\f3q=1,~q>3,\ee
 where  $0<T^{\ast}<\infty$
 is the maximal time of existence of a strong  solution of system \eqref{FNS}.
Subsequently, in \cite{[HL]}, Huang and Li improved \eqref{HLW} to \be \label{HL2}
\lim\sup\limits_{t\rightarrow
T^\star}\left(\|\rho\|_{L^\infty(0,\hs{t};L^\infty(\mathbb{R}^{3}))}+\|u\|_{L^p(0,\hs{t};L^q(\mathbb{R}^{3}))}\right)=\infty, ~~\f2p+\f3q=1,~q>3.\ee
Very recently, authors in \cite{[JWY]} showed that, if $\lambda<3\mu$, \eqref{HL2} can be replaced by
\be \label{JWY}
\lim\sup\limits_{t\rightarrow
T^\star}\left(\|\rho\|_{L^\infty(0,\hs{t};L^\infty(\mathbb{R}^{3}))}+\|\theta\|_{L^p(0,\hs{t};L^q(\mathbb{R}^{3}))}\right)=\infty, ~~\f2p+\f3q=2,~q>3/2.\ee
We would like to \hs{point out} that \hs{the norms in \eqref{HLW}-\eqref{JWY} are scaling-invariant} in the sense of  \eqref{eqscaling}. Inspired by \eqref{serrin1L},    we shall prove the following result
\begin{theorem}
\label{the1.1}
Suppose that $(\rho,u,\theta)$ is the unique strong solution \hs{to (\ref{FNS})}. \hs{Then there} exists a positive constant $\varepsilon$ such that the \hs{solution} $(\rho,u,\theta)$ can be extended   beyond $t=T$ provided that one of the following two conditions holds
  \begin{enumerate}[(1)]
 \item \label{enumerate1}  $\rho \in L^{\infty}(0,T;L^{\infty}(\mathbb{R}^{3}))$,   $u\in L^{p,\infty}(0,T;L^{q,\infty}(\mathbb{R}^{3}))$ and
     \be\label{L11.12}\|u\|_{L^{p,\infty}(0,T;\hs{L^{q,\infty}(\mathbb{R}^{3})})}\leq \varepsilon, ~~\text{with}~~ {2/p}+ {3/q}=1,\ \  q>3;\ee
      \item \label{enumerate2} $\lambda<3\mu,$ $\rho \in L^{\infty}(0,T;L^{\infty}(\mathbb{R}^{3}))$, $\theta\in \hs{L^{p,\infty}}(0,T;L^{q,\infty}(\mathbb{R}^{3}))$ and
    \be\label{L11.13}\|\theta\|_{L^{p,\infty}(0,T;L^{q,\infty}(\mathbb{R}^{3}))}\leq \varepsilon, ~~\text{with}~~ {2/p}+ {3/q}=2,\ \  q>3/2.\ee
  \end{enumerate}
\end{theorem}
\begin{remark}
Theorem \ref{the1.1} \hs{generalizes} both continuation criteria \eqref{HL2} and \eqref{JWY} to enable their space-time directions to be in Lorentz spaces.
\end{remark}

\begin{remark}
 Even if the results in this theorem reduce to the ones \hs{only involving} space direction in Lorentz spaces, they are still new.
\hs{Note that the known best blow-up criteria} based on Lorentz spaces for the   isentropic compressible Navier-Stokes flows were established by Xu and Zhang \cite{[XZ]}
 \be \label{IXZ}
\lim\sup\limits_{t\rightarrow
T^\star}\left(\|\rho\|_{L^\infty(0,t;L^\infty(\mathbb{R}^{3}))}+\|u\|_{L^p(0,t;L^{q,\infty}(\mathbb{R}^{3}))}\right)=\infty, ~~\f2p+\f3q=1,~q>3.\ee
Hence, it seems that Theorem \ref{the1.1} is the first continuation criteria allowing  the time direction to be in Lorentz spaces \hs{for} the compressible fluid.
\end{remark}
\begin{remark}
\hs{Since} $\text{div\,}u \in L^1(0,t;L^\infty(\mathbb{R}^{3}))$  yields $\rho \in L^\infty(0,t;L^\infty(\mathbb{R}^{3}))$,
$\rho \in L^{\infty}(0,t;L^{\infty}(\mathbb{R}^{3}))$ in this theorem can be replaced by $\hs{\text{div\,}u} \in L^{1}(0,t;L^{\infty}(\mathbb{R}^{3}))$. Therefore, this theorem is an extention  of \eqref{HLW}.
\end{remark}
The absolute continuity of norm in
Lorentz space $L^{p,l}(0,T) (l<\infty)$ together with Theorem \ref{the1.1} yields that
\begin{coro}
Suppose $(\rho,u,\theta)$ is the unique strong solution \hs{to (\ref{FNS})}. \hs{Then the strong solution can be extended beyond $t=T$
if one of  the \hs{following two} conditions hold:} for $l<\infty$,
\begin{enumerate}[(1)]
 \item  $\rho \in L^{\infty}(0,T;L^{\infty}(\mathbb{R}^{3}))$,   $u\in L^{p,l}(0,T;L^{q,\infty}(\mathbb{R}^{3}))  \hs{}
     ~~\text{with}~~ {2/p}+ {3/q}=1,\ \  q>3;$
      \item $\lambda<3\mu,$ $\rho \in L^{\infty}(0,T;L^{\infty}(\mathbb{R}^{3}))$, $\theta\in L^{p,l}(0,T;L^{q,\infty}(\mathbb{R}^{3}))  ~~\text{with}~~ {2/p}+ {3/q}=2,\ \  q>3/2.$
  \end{enumerate}
\end{coro}
Owing to the  existence of density in \eqref{FNS},
it seems to be difficult to apply the regularity estimates for \hs{heat equations} and vorticity equations of \eqref{NS} as used in \cite{[Sohr],[KK]} to show Theorem \ref{the1.1}. Our strategy
is to adopt the method introduced by Bosia, Pata and Robinson in \cite{[BPR]} and recently developed in \cite{[JWW]}. In this procedure, the key point  is to derive an estimate in terms of the \hs{following form}, for $\phi(t)\geq0$,
 \be\label{bprjww}
 \f{d}{dt}\phi(t)\leq C\| u\|^{\f{2q}{q-3}}_{L^{q,\infty}(\mathbb{R}^{3})} \phi(t).\ee
We observed that, by means of temperature equation, the critical estimate for \eqref{JWY} in \cite{[JWY]} is that, there exists
$\psi(t)\geq0$ such that
$$
\f{d}{dt}\psi(t)\leq C\| \theta\|^{\f{2q}{2q-3}}_{L^{q}(\mathbb{R}^{3})} \psi(t).
$$
Based on this, invoking the total energy $E=\theta+\f{|u|^{2}}{2}$ as \cite{[HL]}, we can obtain energy estimates in terms of \eqref{bprjww}. Then, Sobolev's  inequality in Lorentz spaces and
boundedness of Riesz transform in Lorentz spaces    allow us to further derive that
$$ \f{d}{dt}\phi(t)\leq C\| u\|^{\f{2q(1-\tau)}{q-3}}_{L^{q,\infty}(\mathbb{R}^{3})}\phi^{1+2\tau}(t).$$
\hs{Subsequently,} we can apply the general
Gronwall inequality in Lemma \ref{2.1} and  Lemma \ref{lemma2.2} to complete the proof.
Finally, we would like to point out that it is unknown whether \eqref{bprjww} holds for the isentropic compressible Navier-Stokes system.

\hs{On the other hand}, authors in \cite{[WWZ]} \hs{recently} obtained the following local regularity criteria in  anisotropic Lebesgue spaces
\hs{for} suitable weak solutions \hs{to} the incompressible Navier-Stokes system \eqref{NS}
\be\label{spe1}
u\in L^{p}_{t}L^{\overrightarrow{q}}_{x}(Q(\varrho)), ~~~~ \text{with} ~~~ \f{2}{p}+\sum^{3}_{j=1}\f{1}{q_{j}}=1, ~~ 1< q_{j}<\infty.\ee
This enlightened  us to  consider the regularity of strong solutions \hs{to} the  compressible Navier-Stokes equations \eqref{FNS} in anisotropic Lebesgue spaces. Before \hs{formulating} our result, we write
\be
\label{xianzhifanwei} \Xi=\{(q_{1},q_{2},q_{3}):\f{1}{q_{i}}+\f{1}{q_{j}}\neq\f{1}{q_{k}}\hs{,}~ 1<q_{i},q_{j},q_{k}< \infty\hs{,~} i,j,k\in\hs{\{1,2,3\}}\}.\ee
\begin{theorem}\label{the1.2}
Let $(\rho,u,\theta)$ be the unique strong solution  \hs{to (\ref{FNS}) and $\overrightarrow{q}=(q_{1},q_{2},q_{3})$. Then,}
  \begin{enumerate}[(1)]
 \item \be\label{genwithout1c1.20}\limsup_{t\rightarrow T^*}\B(
  \|\rho\|_{L^{\infty}(0,\hs{t};L^{\infty}(\mathbb{R}^{3}))}+ \|  u \|_{L^{p }(0,t;L^{\overrightarrow{q} } (\mathbb{R}^{3}) )} \B)= \infty,
   \ee
where the pair $(p ,\,\overrightarrow{q} )$  meets \eqref{xianzhifanwei} and
$$
  \f{2}{p }+\f{1}{q_{1} }+\f{1}{q_{2} }+\f{1}{q_{3} }=1,~~1<q_{i}<\infty\hs{;}$$
  \item \hs{for} $\lambda<3\mu,$
  \be\label{genwithout1c1.21}\limsup_{t\rightarrow T^*}\B(
  \|\rho\|_{L^{\infty}(0,\hs{t};L^{\infty}(\mathbb{R}^{3}))}+ \|  \theta \|_{L^{p }(0,t;L^{\overrightarrow{q}}(\mathbb{R}^{3}) ) } \B)= \infty,
   \ee
where the pair \hs{$(p ,\,\overrightarrow{q} )$  meets} \eqref{xianzhifanwei} and
$$
  \f{2}{p }+\f{1}{q_{1} }+\f{1}{q_{2} }+\f{1}{q_{3} }=2,~~1<q_{i}<\infty.
  $$\end{enumerate}
\end{theorem}
\begin{remark}
In \hs{contrast} with \eqref{spe1}, \eqref{genwithout1c1.20} requires the additional condition \eqref{xianzhifanwei}. The reason is \hs{due to} the application of the boundness of singular integral operator in anisotropic Lebesgue spaces.
\end{remark}
Next, we \hs{demonstrate that} the above proof can be applied to the following nonhomogeneous   incompressible Navier-Stokes equations
\be\left\{\ba\label{NNS}
&\rho_t+\nabla \cdot (\rho u)=0\\
&\rho u_{t}+ \rho u\cdot\ti
u -\Delta  u  +\nabla \Pi=0, \\
&\Div u=0,\\
&u|_{t=0}=u_0,
\ea\right.\ee
The local existence of strong solutions \hs{to} \eqref{NNS} in sense of Proposition  \ref{localwith vacuum nonns} was established  by Choe and   Kim in \cite{[CK]}. In addition, they \hs{showed},
 if $T^{\ast}$ is the finite blow-up time, then
$$\limsup_{t\rightarrow T^*}\|\nabla u\|_{L^{4}(0,t;L^{2}(\mathbb{R}^{3}) )}=\infty.$$
 Kim \cite{[Kim]} extended Serrin type \hs{blow-up} criteria to    system \eqref{NNS} and  showed
\be\label{kim}
\limsup_{t\rightarrow T^*}\|u\|_{L^{p}(0,t;L^{q,\infty}(\mathbb{R}^{3}))}=\infty,~~\text{with} ~~ 2/p+3/q=2,~ \hs{q>3}.
\ee
Our result \hs{concerning}   nonhomogeneous incompressible
 Navier-Stokes equations \eqref{NNS} is
\begin{theorem}
\label{the1.4}
Assume that $(\rho,u)$ is the unique strong
 solution \hs{to \eqref{NNS}.} \hs{Then there} exists a positive constant $\varepsilon$ such that the \hs{solution $(\rho,u)$} can be extended   beyond $t=T$ provided that \hs{one of the following three conditions holds}
  \begin{enumerate}[(1)]
 \item    $u\in L^{p,\infty}(0,T;L^{q,\infty}(\mathbb{R}^{3}))$ and
     \be\label{L11.24}\|u\|_{L^{p,\infty}(0,T;L^{q,\infty}(\mathbb{R}^{3}))}\leq \varepsilon, ~~\text{with}~~ {2/p}+ {3/q}=1,\ \  q>3;\ee
      \item
    $u\in L^{p}(0,T;L^{\overrightarrow{q}}(\mathbb{R}^{3}))  ~~\text{with}~~ {2/p}+ {1/q_{1}}+ {1/q_{2}}+ {1/q_{3}}=2,\ \  1<q_{i}<\infty \hs{;}$
    \item\label{enumeratethe1.43}
    $\nabla u\in L^{2}(0,T;L^{3}(\mathbb{R}^{3})).$
  \end{enumerate}
\end{theorem}
\begin{remark}
Theorem \ref{the1.4} is an improvement of \eqref{kim}.
In addition, collecting \eqref{kim}, \eqref{enumeratethe1.43} in Theorem \ref{the1.4} and
 Gagliardo-Nirenberg  inequality, we obtain an   alternative condition of \eqref{kim} below
 \be\label{de v}\limsup_{t\rightarrow T^*}  \|\nabla u \|_{L^{p}(0,t;L^{q}(\mathbb{R}^{3}))}
  = \infty,~~{\text with}~~\f{2}{p}+\f{3}{q}=2,~ 3/2<q\leq 3 .\ee
\end{remark}
\begin{remark}
Very recently, Wang \cite{[Wang]} proved \eqref{kim}
 \hs{is also valid} for nonhomogeneous incompressible
 Navier-Stokes equations with heat conducting. It is worth remarking that \hs{Theorem \ref{the1.4} also holds} for this system.
We leave this to the interested readers.
\end{remark}

The remaining part of this paper \hs{is} organized as follows.
In the next section, some \hs{materials} \hs{involving}
Lorentz spaces and anisotropic Lebesgue spaces are listed.
\hs{Several auxiliary lemmas} are also given.
Then, we recall the local \hs{well-posedness}
of equations \eqref{FNS} and \eqref{NNS}, respectively.
 \hs{Section 3 is devoted to proving Theorem \ref{the1.1} and Theorem \ref{the1.2}.}
 In \hs{Section 4}, we are concerned with lifespan
 of nonhomogeneous
incompressible Navier-Stokes equations.

\section{Preliminaries} \label{section2}

First, we introduce some notations used in this paper.
 For $p\in [1,\,\infty]$, the notation $L^{p}(0,T;X)$ stands for the set of measurable functions $f(x,t)$ on the interval $(0,T)$  with values in $X$ and $\|f(\cdot,t)\|_{X}$ belonging to $L^{p}(0,T)$.
 The classical Sobolev space $W^{k,\hs{p}}(\mathbb{R}^{3})$ is equipped with the norm $\|f\|_{W^{k,\hs{p}}(\mathbb{R}^{3})}=\sum\limits_{\wg{|\alpha| =0}}^{k}\|D^{\alpha}f\|_{L^{\hs{p}}(\mathbb{R}^{3})}$. \hs{$|E|$ represents the $n$-dimensional Lebesgue measure of a set $E\subset \mathbb{R}^{n}$.}
 We will use the summation convention on repeated indices.
 $C$ is an absolute constant which may be different from line to line unless otherwise stated in this paper.

A function $f$ belongs to  the homogeneous Sobolev spaces $D^{k,l}\hs{(\mathbb{R}^3)}$
if $
u\in L^1_{\rm{loc}}(\mathbb{R}^3)$ \hs{such that} $\|\nabla^k u \|_{L^l\hs{(\mathbb{R}^3)}}<\infty.$
For simplicity,   we write
$$   \ H^k=W^{k,2}(\mathbb{R}^3), \ D^k=D^{k,2}(\mathbb{R}^3).$$

 We denote \hs{$G$ as} the  effective viscous flux, that is, \be\label{evf} G=(2\mu+\lambda)\text{div\,}u-P.\ee
The notation $\dot{v}=v_t+u\cdot\nabla v$  stands for material derivative.
Note that
$$
\Delta G=\text{div}(\rho \dot{u})~~~~ \text{and }~~~~ -\mu\Delta(\text{curl} u)=\nabla\times(\rho\dot{u} ).
$$
It is well-known that
\be \label{ntle}
\ba
\|\nabla G\|_{L^{p}}\leq C\|\rho \dot{u}\|_{L^{p}},   \ \|\nabla \text{curl} u\|_{L^{p}}\leq C\|\rho \dot{u}\|_{L^{p}},\forall p\in \hs{(1,\infty)}.
\ea
\ee

\subsection{Some basic facts on Lorentz spaces}
Next, we present some basic facts on Lorentz spaces.
For $p,q\in[1,\infty]$, we define
$$
\|f\|_{L^{p,q}(\mathbb{R}^{3})}=\left\{\ba
&\B(p\int_{0}^{\infty}\alpha^{q}|\{x\in \hs{\mathbb{R}^{3}}:|f(x)|>\alpha\}|^{\f{q}{p}}\f{d\alpha}{\alpha}\B)^{\f{1}{q}} , ~~~q<\infty, \\
 &\sup_{\alpha>0}\alpha|\{x\in \hs{\mathbb{R}^{3}}:|f(x)|>\alpha\}|^{\f{1}{p}} ,~~~q=\infty.
\ea\right.
$$	
Furthermore,
$$
L^{p,q}(\mathbb{R}^{3})=\big\{f: f~ \text{is  a measurable function  on}~ \hs{\mathbb{R}^{3}} ~\text{and} ~\|f\|_{L^{p,q}(\mathbb{R}^{3})}<\infty\big\}.
$$ 		
Similarly, one can  define
Lorentz spaces $L^{p,q}(0,T;X)$ in time for $ p\leq q \leq\infty$. $f\in L^{p,   q}(0,T;X)$ means that $\|f\|_{L^{p,q}(0,T;X)}<\infty$, where
$$\|f\|_{L^{p,q}(0,T;X)}=\left\{\ba
&\B( p \int_{0}^{\infty}\alpha^q|\{t\in(0,T)
:\|f(t)\|_{X}>\alpha\}|^{\f{q}{p}}\f{d\alpha}{\alpha}\B)^{\f{1}{q}} , ~~~q<\infty, \\
 &\sup_{\alpha>0}\alpha|\{t\in(0,T)
:\|f(t)\|_{X}>\alpha\}|^{\f{1}{p}} ,~~~q=\infty.\ea\right.
$$
We list the properties of Lorentz spaces  as follows.
\begin{itemize}
\item
H\"older's inequality in Lorentz spaces  \cite{[Neil]}
\be\label{hiL}\ba
 &\|fg\|_{L^{r,s}(\mathbb{R}^{n})}\leq \|f\|_{L^{r_{1},s_{1}}(\mathbb{R}^{n})}\|g\|_{L^{r_{2},s_{2}}(\mathbb{R}^{n})},
\\
&\f{1}{r}=\f{1}{r_{1}}+\f{1}{r_{2}},~~\f{1}{s}=\f{1}{s_{1}}+\f{1}{s_{2}}.
\ea\ee
\item Interpolation characteristic of Lorentz spaces \cite{[BL]}
\be\label{Interpolation characteristic}
(L^{p_{0},q_{0}}(\mathbb{R}^{n}),L^{p_{1},q_{1}}(\mathbb{R}^{n}))_{\alpha,q}=L^{p,q}(\mathbb{R}^{n})
~~~~\text{with}~~~ \f{1}{p}=\f{1-\alpha}{p_{0}}+\f{\alpha}{p_{1}},~0<\alpha<1.\ee

\item
Boundedness of Riesz Transform in Lorentz spaces \cite{[CF]}
\be\|R_{j}f\|_{L^{p,q}(\mathbb{R}^{n})}\leq C\| f\|_{L^{p,q}(\mathbb{R}^{n})},~1<p<\infty.\label{brl}\ee

\item

The Lorentz spaces $L^{p,r}$ increase as the exponent $r$ increases \cite{[Grafakos],[Maly]}

For $1\leq p\leq\infty$ and $1\leq r_{1}<r_{2}\leq\infty,$
\be\label{Lorentzincrease}
\|f\|_{L^{p,r_{2}}(\mathbb{R}^{n})}\leq \B(\f{r_{1}}{p}\B)^{\f{1}{r_{1}}-\f{1}{r_{2}}}\|f\|_{L^{p,r_{1}}(\mathbb{R}^{n})}.
\ee

\item  Sobolev's  inequality in Lorentz spaces \cite{[Neil],[Tartar]}
\be\label{sl}
\|f\|_{L^{\f{np}{n-p},p}(\mathbb{R}^{n})}\leq \hs{C} \|\nabla f\|_{L^{p}(\mathbb{R}^{n})}~~\text{with}~~ 1\leq p<n.\ee
\end{itemize}

\subsection{Some  \hs{materials} on anisotropic Lebesgue \hs{spaces}}

\hs{Let $x=(x_1,\,x_2,\,x_3)\in\mathbb{R}^{3}, \overrightarrow{q}=(q_1,\,q_2,\,q_3)$, where $q_i \in [1,\,\infty], 1\leq i \leq 3$.} A function $\hs{f(x)}$ belongs to  the anisotropic Lebesgue \hs{space $L^{\overrightarrow{q}}(\mathbb{R}^{3})$}
if
  $$\hs{\|f\|_{L^{\overrightarrow{q}}(\mathbb{R}^{3})}=\|f\|_{L_{1}^{q_{1}}L_{2}^{q_{2}}
L_{3}^{q_{3}}(\mathbb{R}^{3})}=
  \B\|\big\|\|f\|_{L_{x_1}^{q_{1}}(\mathbb{R})}\big\|_{L_{x_2}^{q_{2}}(\mathbb{R})}\B\|_{L_{x_3}^{q_{3}} (\mathbb{R})}}<\infty.  $$
The study of
 anisotropic Lebesgue spaces  first appears in
  Benedek and   Panzone \cite{[BP]}.

We list the properties of \hs{anisotropic Lebesgue spaces} as follows.
\begin{itemize}
\item The H\"older inequality in anisotropic Lebesgue spaces \cite{[BP]}
 \be\label{HIAL}
 \|fg\|_{L^{1}(\hs{\mathbb{R}^{3}})}\leq \|f \|_{L^{\overrightarrow{r}}(\hs{\mathbb{R}^{3}})} \| g\|_{L^{\overrightarrow{s}}(\hs{\mathbb{R}^{3}})}\hs{~~\text{with}~~} \hs{\f{1}{\overrightarrow{r}}+\f{1}{\overrightarrow{s}}=1.}
 \ee

\item The interpolation  inequality in anisotropic Lebesgue spaces \cite{[BP]}
        \be\label{interi}
\|f\|_{L^{\overrightarrow{r}}(\mathbb{R}^{3}) }
\leq \|f\|^{\alpha}_{L^{\overrightarrow{s}}(\mathbb{R}^{3}) }\|f\|^{1-\alpha}_{L^{\overrightarrow{t}}(\mathbb{R}^{3}) }~~\text{with}~~\f{1}{\hs{\overrightarrow{r}}}=\f{\alpha}{\hs{\overrightarrow{s}}}+\f{1-\alpha}{\hs{\overrightarrow{t}}},~0<\alpha<1.
\ee
\item
Boundedness of  singular integral operator in anisotropic Lebesgue spaces \cite{[RRT]}
\be\|R_{j}f\|_{L^{\overrightarrow{q}}(\hs{\mathbb{R}^{3}})}\leq C\| f\|_{L^{\overrightarrow{q}}(\hs{\mathbb{R}^{3}})},~1<\hs{\overrightarrow{q}}<\infty.
\label{bsa}\ee

\item  Sobolev's  inequality in anisotropic Lebesgue spaces  \cite{[Zheng],[GCS]} \\
 \hs{For} $q_{1},q_{2},q_{3}\in[2,\infty)$ and \wred{$0\leq \hs{\sum_{i} \f{1}{q_i}}-\f12\leq 1$}\hs{,}
\begin{align}\label{zc}\|f\|_{\wred{L^{\overrightarrow{q}}(\mathbb{R}^3)}}
 &\leq C
\|\partial_{1}f\|^{\f{q_{1}-2}{2q_{1}}}_{L^{2}(\mathbb{R}^{3})}\|\partial_{2}f\|^{\f{q_{2}-2}{2q_{2}}}_{L^{2}(\mathbb{R}^{3})}
\|\partial_{3}f\|^{\f{q_{3}-2}{2q_{3}}}_{L^{2}(\mathbb{R}^{3})}\|f\|^{\hs{\sum_{i} \f{1}{q_i}}-\f12}_{L^{2}(\mathbb{R}^{3})}
\nonumber\\&\leq C\|\nabla f\|_{L^{2}(\mathbb{R}^{3})}^{\f32-{\hs{\sum_{i} \f{1}{q_i}}}}\|f\|^{\hs{\sum_{i} \f{1}{q_i}}-\f12}_{L^{2}(\mathbb{R}^{3})}.
\end{align}
\end{itemize}

\subsection{Auxiliary lemmas  }
The following powerful
Gronwall lemma obtained by  Bosia,   Pata and   Robinson in \cite{[BPR]} \wg{helps} us to prove Theorem \ref{the1.1}.
  \begin{lemma}[\cite{[BPR]}]\label{2.1}
 Let $\phi$ be a measurable positive function
  defined on the interval $[0,T]$.
  Suppose that there  exists
  $\tau_{0}>0$ such that for all $0<\tau<\tau_{0}$ and a.e. $t\in[0,T]$, $\phi$ satisfies the inequality
 $$\f{d}{dt}\phi\leq \mu\lambda^{1-\tau}\phi^{1+2\tau},$$
 where   $0 <\lambda \in L^{1,\infty}(0,T)$ and $\mu> 0$  with
 $$\mu\|\lambda\|_{L^{1,\infty}(0,T)}<\f12.$$
 Then $\phi$ is bounded on $[0,T]$.
 \end{lemma}
To apply the above lemma, we require the following fact.
\begin{lemma}[\cite{[JWW]}]\label{lemma2.2}
Assume that the pair $(p,q)$ satisfies $\f{2}{p}+\f{3}{q}=a$  with
$a,q \geq1 $ and $p>0$. Then, for every $\tau\in[0,1]$ and given $b,c_0\geq1$, there  exist $p_{\tau} > 0$ and $\min\{q,b\}\leq q_{\tau}\leq\max\{q,b\}$ such that
\be\left\{\ba\label{ro}
&\f{2}{p_{\tau}}+\f{3}{q_{\tau}}=a, \\
&\f{p_{\tau}}{q_{\tau}}=\f{p\big(1-\tau\big)}{q}+\f{c_0\tau}{b}.
\ea\right.\ee
\end{lemma}

 Finally, we recall
the local well-posedness
  of   strong solutions to the full
compressible Navier-Stokes equations \eqref{FNS}
and nonhomogeneous incompressible Navier-Stokes equations
\eqref{NNS}.
 \begin{prop}[\cite{[CK]}]\label{localwith vacuum}
Suppose $ u_0, \theta_0 \in D^1(\mathbb{R}^{3})\cap D^2(\mathbb{R}^{3})$ and
\[\rho_0 \in W^{1,q}(\mathbb{R}^{3}) \cap H^1(\mathbb{R}^{3})\cap L^1(\mathbb{R}^{3})\]
for some $q\in (3,6]$.
If $\rho_0$ is nonnegative and the initial data \hs{to \eqref{FNS}} satisfy the compatibility condition
$$\ba
&-\mu\Delta  u_0-(\mu+\lambda)\nabla\emph{div}\,u_0 +\nabla \hs{P}(\rho_0,\theta_0) = \sqrt{\rho_0} g_1\\
&\Delta\theta_0+\frac{\mu}{2}|\nabla  u_0 +(\nabla u_0)^{\text {tr}}|^2
+ \lambda(\hs{\emph{div}\,u_0})^2 =\sqrt{\rho_0} g_2
\ea$$
for    vector fields  $g_1,g_2\in L^2(\mathbb{R}^{3})$.
Then there \hs{exist} $T\in (0,\infty]$ and \hs{a} unique solution \hs{to \eqref{FNS}},  satisfying
$$\ba
&(\rho, u,\theta)\in C([0,T );L^1\cap H^1(\mathbb{R}^{3})\cap W^{1,q}(\mathbb{R}^{3})) \times C([0,T);D^1\cap D^2 )\times L^2([0,T);D^{2,q}) \\
&(\rho_t, u_t,\theta_t)\in C([0,T );L^2\cap L^q(\mathbb{R}^{3}))\times L^2([0,T); D^1)\times L^2([0,T);D^1) \\
&(\rho^{1/2} u_t,\rho^{1/2}\theta_t)\in L^\infty([0,T);L^2(\mathbb{R}^{3})) \times L^\infty([0,T);L^2).
\ea$$
\end{prop}
\begin{prop}[\cite{[CCK]}]\label{localwith vacuum nonns}
Suppose $ u_0 \in D^1(\mathbb{R}^{3})\cap D^2(\mathbb{R}^{3})$ and
\[\rho_0 \in W^{1,q}(\mathbb{R}^{3}) \cap H^1(\mathbb{R}^{3})\cap L^1(\mathbb{R}^{3})\]
for some $q\in (3,6]$.
If $\rho_0$ is nonnegative and the initial data \hs{to \eqref{NNS}} satisfy the compatibility condition
$$
 -\Delta  u_0 +\nabla \Pi_{0}  = \sqrt{\rho_0} g
$$
for    vector \hs{field}  $g  \in L^2(\mathbb{R}^{3})$.
Then there \hs{exist} $T\in (0,\infty]$ and \hs{a} unique solution \hs{to \eqref{NNS}},  satisfying
$$\ba
&(\rho, u )\in C([0,T );L^1(\mathbb{R}^{3})\cap H^1(\mathbb{R}^{3})\cap W^{1,q}(\mathbb{R}^{3})) \times C([0,T);D^1\cap D^2 ) \\
&(\rho_t, u_t )\in C([0,T );L^2(\mathbb{R}^{3})\cap L^q(\mathbb{R}^{3}))\times L^2([0,T); D^1)  \\
& \rho^{1/2} u_t \in L^\infty([0,T);L^2(\mathbb{R}^{3})).
\ea$$
\end{prop}
\section{ Full compressible
Navier-Stokes equations }
\label{sec3}
\setcounter{section}{3}\setcounter{equation}{0}
\subsection{Continuation criteria in Lorentz spaces }

This section is devoted to proving  Theorem \ref{the1.1}.
\begin{proof}[Proof of \eqref{enumerate1} in Theorem \ref{the1.1}]
Multiplying $\eqref{FNS}_{2}$ by $u_{t}$, integrating over $\mathbb{R}^{3}$, we can write
\be\ba\label{3.10}
&\f{1}{2}\f{d}{dt}\int\B[\mu|\nabla u|^{2}+(\lambda+\mu)(\text{div\,}u)^{2}\B]dx+\int \rho |\dot{u}|^{2}dx\\=&\int \rho \dot{u}\cdot(u\cdot\nabla u)dx+\int P\text{div\,}u_{t}dx\\=&I+II.
\ea
 \ee
In view of the Young inequality, we see that
 \be\ba\label{3.11}
I\leq\f{1}{4}\int \rho |\dot{u}|^{2}dx+C\int |u|^{2}|\nabla u|^{2}  dx.
\ea
 \ee
 By means of   the effective viscous flux \eqref{evf}, we calculate
 \be\ba\label{3.12}
II&=\f{d}{dt}\int P\text{div\,}udx-\int P_{t}\text{div\,}udx\\&=\frac{d}{dt}\int P\text{div\,}u-\frac{1}{2(2\mu+\lambda)}
\frac{d}{dt}\int P^2
-\frac{1}{2\mu+\lambda}\int P_tG
\\&=II_{1}+II_{2}+II_{3}.
\ea
 \ee
We define total energy
$$E=:\theta+\f{|u|^{2}}{2}.$$
It is clear that, for $\alpha\geq1$,
\be\label{citaE}
0\leq\theta^{\alpha}\leq E^{\alpha}.
\ee
From \eqref{FNS}, we know that
\be\label{eqofE}
(\rho E)_{t}-\f{\kappa}{c_{v}}\text{div}(\nabla E) =\text{div}F,
\ee
where
$$
F=-\rho Eu+\f{\mu-\kappa c_{v}}{2}\nabla(|u|^{2})+\mu u\cdot\nabla u+\lambda u\text{div} u-Pu.
$$
Performing some direct calculations yields that
  \be\ba\label{db}
|F|&\leq C(|\rho Eu|+|u||\nabla u|+|Pu|)\\
 &\leq C(|\rho Eu|+|u||\nabla u|+R|\rho\theta u|)\\
&\leq C(|\rho Eu|+|u||\nabla u|).
 \ea \ee
 According to $P=R\rho\theta=R(\rho E-\rho\f{|u|^{2}}{2})$
and \eqref{eqofE},
we note that
\be  \ba\label{3.13}
II_{3}=&-\frac{R}{2\mu+\lambda}\int\ (\rho E)_t G+
\frac{1}{2\mu+\lambda}\int\ \left(\frac{\rho|u|^2}{2}\right)_t G\\=&-\frac{R}{2\mu+\lambda}\int\text{div}( F+ \f{\kappa}{c_{v}} \nabla E )  G+
\frac{R}{2\mu+\lambda}\int\ \left(\frac{\rho|u|^2}{2}\right)_t G\\=&\frac{R}{2\mu+\lambda}\int\   (F+ \f{\kappa}{c_{v}} \nabla E)  \nabla G+
\frac{R}{2\mu+\lambda}\int\ \left(\frac{\rho|u|^2}{2}\right)_t G \\=&II_{31}+II_{32} . \ea
\ee
In light of
\eqref{db}, the Young inequality and  \eqref{ntle}, we have
\be\ba\label{wwwy3.8}
II_{31}&\leq C \int\  ( |F| + |\nabla E|  ) | \nabla G|\\&\leq C \int\  ( |\rho Eu|+|u||\nabla u|+ |\nabla E|)  \nabla G \\
&\leq \eta\|\nabla G\|^{2}_{L^2}+C\int (\rho^{2} E^{2}|u|^{2}+|\nabla E|^{2}+|u|^{2}|\nabla u|^{2})  dx\\&\leq \eta C\|\rho \dot{u}\|^{2}_{L^2}+C\int (\rho^{2} E^{2}|u|^{2}+|\nabla E|^{2}+|u|^{2}|\nabla u|^{2})  dx\hs{.}
\ea\ee
It is straightforward to check
\be \ba\label{3.15}
II_{32}=&\frac{R}{2\mu+\lambda}\int\ \frac{\rho_t|u|^2}{2} G+
\frac{R}{2\mu+\lambda}\int\ \rho u\cdot u_t G\\
=&II_{321}+II_{322}.
\ea
\ee
Making use of $ \rho_t=-\text{div\,}(\rho u)$,
integrating by parts, the Young inequality, \eqref{citaE} and \eqref{ntle}, we see that
\be\ba\label{3.16}
II_{321}=&-\frac{R}{2\mu+\lambda}\int
\ \frac{\text{div\,}(\rho u)|u|^2}{2} G\\
=&\frac{R}{2\mu+\lambda}\int\ \rho u\cdot\nabla u\cdot u G
+\frac{R}{2\mu+\lambda}\int\ \frac{\rho u|u|^2}{2}\cdot\nabla G\\
\leq& C\int\ \rho |u\cdot\nabla u|^{2}+C\int\ \rho |u|^{2}| G|^{2}+C\int\ \rho |u|^6  +\eta\int| \nabla G|^{2}\\
\leq& C\int\ \rho |u\cdot\nabla u|^{2}+C\int\ \rho |u|^{2}(|\nabla u|^{2}+R\rho^{2}\theta^{2})+C\int  \rho |u|^{2} (\theta^{2}+E^2) +\eta C\int|\rho \dot{u}|^{2}\\
 \leq& C\int\rho   |u|^{2} |\nabla u|^{2}+C\int\ \rho |u|^{2} E^{2} +\eta C\int|\rho \dot{u}|^{2}\hs{,}
\ea\ee
where we have used the fact
\be\label{nxtlc}|G|\leq C(|\nabla u|+|\theta\rho|)\leq C(|\nabla u|+|E\rho|).\ee
\hs{Utilizing} $\dot{u}=u_t+u\cdot\nabla u$, the Young inequality, and \eqref{nxtlc}, we arrive at that
\be\ba\label{3.18}
II_{322}&=\frac{R}{2\mu+\lambda}\int\ \rho u\cdot (\dot{u}-u\cdot\nabla u) G\\&=\frac{R}{2\mu+\lambda}\int\ \rho u\cdot \dot{u}G-\int\ \rho u u\cdot\nabla u  G\\
&\leq \eta\int \rho |\dot{u}|^{2}dx+C\int \rho |  u |^{2}|G |^{2}dx +\int\ \rho |u|^2 |\nabla u|^2
\\
&\leq \eta\int \rho |\dot{u}|^{2}dx+C\int \rho   |u|^{2} |\nabla u|^{2}+C\int\ \rho |u|^{2} E^{2}.
 \ea\ee
Plugging  \eqref{3.16} and \eqref{3.18} into \eqref{3.15}, we deduce
 \be  \ba\label{3.19}
II_{32}\leq& \eta(1+C)\int \rho |\dot{u}|^{2}dx+C\int   |u|^{2} |\nabla u|^{2} +C\int\ \rho |u|^{2} E^{2}.
\ea\ee
It follows    from \eqref{3.13}, \eqref{wwwy3.8} and \eqref{3.19} that
 \be  \ba\label{3.20}
&II_{3}
\leq &C\eta\int \rho |\dot{u}|^{2}dx+C\int \rho  |u|^{2}|\theta|^{2}dx +C\int  |u|^{2}|\nabla u|^{2} +C \int|\nabla E|^{2}.
\ea\ee
Combining \eqref{3.10}, \eqref{3.11}, \eqref{3.12} and \eqref{3.20} and choosing $\eta$ sufficiently small lead  us to the estimate
\be\ba\label{k2}
& \hs{\f{1}{2} \f{d}{dt}}\int\B[\mu|\nabla u|^{2}+(\lambda+\mu)(\text{div\,}u)^{2}+ \frac{1}{ (2\mu+\lambda)} P^2-2P\text{div\,}u\B]dx+ \f{1}{2} \int \rho |\dot{u}|^{2}dx\\
&\leq C_{1}\int \rho  |u|^{2}|E|^{2}dx +C_{1}\int |u|^{2}|\nabla u|^{2} +C_{1}  \int|\nabla E|^{2}.
\ea\ee
To control the first term of right hand side of the latter \hs{relation}, we use \eqref{eqofE}. More  precisely,
multiplying \eqref{eqofE} by $E$, integrating by parts and invoking $(\rho E)_{t}E=\f{1}{2}(\rho E^{2})_{t}-\f12 \hs{\textrm{div}}(\rho u)$, we find
\be\ba\label{k21}
\f{1}{2} \f{d}{dt}\int \rho E^{2} + \f{\kappa}{c_{v}} \int |\nabla E|^{2}dx&= \f{1}{2}\int \text{div\,}(\rho u)E^{2}+\int E\text{div\,}F\\&= -\f{1}{2}\int (\rho u)\nabla E^{2}-\int \hs{F\nabla E}.
\ea\ee
The Young inequality \hs{helps} us to obtain
$$\ba
&-\f{1}{2}\int (\rho u)\nabla E^{2}\leq C\int \rho|u|^{2}|E|^{2}+ \f{\kappa}{4c_{v}} \int  |\nabla E|^{2}dx,
\\
&\int \hs{F\nabla E}\leq  C\int |F|^{2}  + \f{\kappa}{4c_{v}} \int   |\nabla E|^{2}dx.
\ea$$
As a consequence,
\be\ba\label{k1}
\f{1}{2} \f{d}{dt}\int \rho E^{2} + \f{\kappa}{2c_{v}} \int |\nabla E|^{2}dx\leq&
C\int\rho |u|^{2}|E|^{2} +C\int |F|^{2}  \\
\leq& C \int    |u|^{2}|\nabla u|^{2}dx +C \int \rho|u|^{2}|E|^{2} .
\ea\ee
Multiplying \hs{both sides of \eqref{k1}} by $\f{c_{v}(2C_{1}+C_{2})}{\kappa}$ and adding it \hs{to}  \eqref{k2}, we know that
\be\ba\label{k221}
  \f{d}{dt}\hs{\Phi(t)} + \f{1}{2} \int \rho |\dot{u}|^{2}dx
  +\f{C_{2}}{2} \int|\nabla E|^{2}
 \leq C \int \rho  |u|^{2}|E|^{2}dx +C \int |u|^{2}|\nabla u|^{2},
\ea\ee
where
$$
\Phi(t)=\f{1}{2}\int\B[\mu|\nabla u|^{2}+(\lambda+\mu)(\text{div\,}u)^{2}+
\frac{1}{ (2\mu+\lambda)} P^2-2P\text{div\,}u
+\f{c_{v}(2C_{1}+C_{2})}{\kappa}\rho E^{2}\B]dx.
$$
According to  $P=R\rho\theta$,
we can choose   $C_{2}$ sufficiently large to make sure that
\be\label{key1}
\mu|\nabla u|^{2}+\f{c_{v}(2C_{1}+C_{2})}{\kappa}\rho E^{2} -2P\text{div\,}u\geq
\f\mu2|\nabla u|^{2}+\rho E^{2}.
\ee
Thus, $\hs{\Phi(t)} \geq0. $

With the help of the H\"older inequality \eqref{hiL} or interpolation characteristic \eqref{Interpolation characteristic}
and Sobolev embedding \eqref{sl} in Lorentz spaces, we get
\be\ba\label{sampleinterplation}
\| \hs{\rho} E^{2}\|_{L^{\f{q}{q-2},1} (\mathbb{R}^{3})}&\leq
\| \sqrt{\rho} E \|^{ \f6q}_{L^{6,2}(\mathbb{R}^{3})}\|  \sqrt{\rho}E\|^{2-\f{6}{q}}_{L^{ 2} (\mathbb{R}^{3})}\\&\leq
C\|   E \|^{ \f6q}_{L^{6,2}(\mathbb{R}^{3})}\|  \sqrt{\rho}E\|^{2-\f{6}{q}}_{L^{ 2} (\mathbb{R}^{3})}\\&\leq C
\| \nabla E \|^{ \f6q}_{L^{2}(\mathbb{R}^{3})}
\| \sqrt{\rho}E\|^{2-\f{6}{q}}_{L^{ 2} (\mathbb{R}^{3})}.
\ea\ee
Therefore,
  H\"older's inequality \eqref{hiL}, \eqref{sampleinterplation} and the Young inequality imply that
\be\ba\label{i11}
\int \rho  |u|^{2}|E|^{2}dx&\leq \| |u|^{2}\|_{L^{\f{q}{2},\infty}(\mathbb{R}^{3})}\|  \rho  E^{2}\|_{L^{\f{q}{q-2},1} (\mathbb{R}^{3})}
\\&\leq \| u\|^{2}_{L^{\f{q}{2},\infty}(\mathbb{R}^{3})}\|   \sqrt{\rho} E\|^{2}_{L^{\f{2q}{q-2},2} (\mathbb{R}^{3})}\\
&\leq C\| u\|^{2}_{L^{q,\infty}(\mathbb{R}^{3})}\| \nabla  E \|^{\f6q}_{L^{2}(\mathbb{R}^{3})}\|\sqrt{\rho} E\|^{2-\f{6}{q}}_{L^{ 2} (\mathbb{R}^{3})}\\
&\leq C\|  u\|^{\f{2q}{q-3}}_{L^{q,\infty}(\mathbb{R}^{3})}\|  \sqrt{\rho} E\|^{2}_{L^{ 2} (\mathbb{R}^{3})}
+\f{C_{2} }{8}\| \nabla  E \|^{2}_{L^{2}(\mathbb{R}^{3})}.
\ea\ee
 Employing boundedness of Riesz Transform in Lorentz spaces
\eqref{brl}, \eqref{evf}, Sobolev inequality \eqref{sl} in Lorentz spaces and \eqref{ntle}, we infer that
$$\ba
\|\nabla u\|_{L^{6,2}(\mathbb{R}^{3})}\leq& C (\|\hs{\textrm{div}}\,u\|_{L^{6,2}(\mathbb{R}^{3})}+\|\hs{\textrm{curl}}\,u\|_{L^{6,2}(\mathbb{R}^{3})})\\
\leq& C (\|G\|_{L^{6,2}(\mathbb{R}^{3})}+\|P\|_{L^{6,2}(\mathbb{R}^{3})}+\|\hs{\textrm{curl}}\, u\|_{L^{6,2}(\mathbb{R}^{3})})\\
\leq& C (\|G\|_{L^{6,2}(\mathbb{R}^{3})}+\|\rho E\|_{L^{6,2}(\mathbb{R}^{3})}+\|\hs{\textrm{curl}}\, u\|_{L^{6,2}(\mathbb{R}^{3})})\\
\leq& C (\|\nabla G\|_{L^{2,2}(\mathbb{R}^{3})}+\|\nabla  E\|_{L^{2,2}(\mathbb{R}^{3})}+\|\nabla \hs{\textrm{curl}}\, u\|_{L^{2,2}(\mathbb{R}^{3})})\\
\leq& C (\|\rho \dot{u}\|_{L^{2}(\mathbb{R}^{3})}+\|\nabla  E\|_{L^{2}(\mathbb{R}^{3})} ).
\ea$$
With this in hand, modifying the proof of \eqref{i11}, we conclude that
\be\ba\label{i12}
C \int  |u|^{2}|\nabla u|^{2}&\leq
\| |u|^{2}\|_{L^{\f{q}{2},\infty}(\mathbb{R}^{3})}\|  |\nabla u|^{2}\|_{L^{\f{q}{q-2},1} (\mathbb{R}^{3})}\\&\leq
 \|  u\|^{2}_{L^{q,\infty}(\mathbb{R}^{3})}\|  \nabla u\|^{2}_{L^{\f{2q}{q-2},2} (\mathbb{R}^{3})}\\
&\leq C\|  u\|^{2}_{L^{q,\infty}(\mathbb{R}^{3})}\| \nabla  u \|^{2-\f6q}_{L^{2}(\mathbb{R}^{3})}\| \nabla u\|^{\f{6}{q}}_{L^{ 6,2} (\mathbb{R}^{3})}\\
&\leq C\| u\|^{2}_{L^{q,\infty}(\mathbb{R}^{3})}\| \nabla  u \|^{2-\f{6}{q}}_{L^{2}(\mathbb{R}^{3})} (\|\rho \dot{u}\|_{L^{2}(\mathbb{R}^{3})}+\|\nabla  E\|_{L^{2}(\mathbb{R}^{3})} )^{\f6q} \\
&\leq C\| u\|^{\f{2q}{q-3}}_{L^{q,\infty}(\mathbb{R}^{3})}
\|  \nabla u\|^{2}_{L^{ 2} (\mathbb{R}^{3})}+
 \f{1}{4}\|\rho \dot{u}\|_{L^{2}(\mathbb{R}^{3})}
+\f{C_{2} }{8}\|\nabla  E\|_{L^{2}(\mathbb{R}^{3})} \hs{.}
\ea\ee
Substituting \eqref{i11} and \eqref{i12} into \hs{\eqref{k221}, we have}
\be\ba
&   \f{d}{dt}\hs{\Phi(t)}+ \f{1}{4} \int \rho |\dot{u}|^{2}dx
+\f{C_{2} }{4} \int|\nabla E|^{2}dx
\\\leq& C\| u\|^{\f{2q}{q-3}}_{L^{q,\infty}(\mathbb{R}^{3})}(\|\sqrt{  \rho} E\|^{2}_{L^{ 2} (\mathbb{R}^{3})}
+\|  \nabla u\|^{2}_{L^{ 2} (\mathbb{R}^{3})}).
\ea\ee
Owing to \eqref{key1} and     dropping the unnecessary terms, we infer
\be \label{wwwy3.24}
\f{d}{dt}\Phi(t)\leq C  \|  u\|^{\f{2q}{q-3}}_{L^{q,\infty}(\mathbb{R}^{3})}\Phi(t).\ee
The interpolation characteristic \eqref{Interpolation characteristic}, Sobolev's embedding inequality \eqref{sl}, \hs{Lemma}  \ref{lemma2.2} with $a=1, b=6, c_{0}=4$ and \eqref{Lorentzincrease} allow us to \hs{deduce the estimate}
\be\label{2.2000}\ba
 \| u\|^{p_{\tau}}_{L^{q_{\tau},\infty}(\mathbb{R}^{3})}&\leq
 \| u\|^{p(1-\tau)}_{L^{q ,\infty}(\mathbb{R}^{3})} \|  u\|^{4\tau}_{L^{6,\infty}\hs{(\mathbb{R}^{3})}}\\&\leq C
 \| u\|^{p(1-\tau)}_{L^{q ,\infty}(\mathbb{R}^{3})} \| u\|^{4\tau}_{L^{6,2 }(\mathbb{R}^{3})}\\
 &\leq C
 \| u\|^{p(1-\tau)}_{L^{q ,\infty}(\mathbb{R}^{3})} \|\nabla u \|^{4\tau}_{L^{2 }(\mathbb{R}^{3})}\\
 &\leq C
 \| u\|^{p(1-\tau)}_{L^{q ,\infty}(\mathbb{R}^{3})} \Phi(t)^{2\tau},
\ea\ee
where \eqref{key1} was used again.

\hs{Since the pair $(p_{\tau}, q_{\tau})$ also meets $2/p_{\tau}+3/q_{\tau}=1$, we insert \eqref{2.2000} into \eqref{wwwy3.24} to obtain}
$$
\f{d}{dt}\Phi(t) \hs{\leq C  \|  u\|^{p_{\tau}}_{L^{q_{\tau},\infty}(\mathbb{R}^{3})}\Phi(t)} \leq C  \| u\|^{p(1-\tau)}_{L^{q ,\infty}(\mathbb{R}^{3})} \Phi(t)^{1+2\tau}.
$$
At this stage, Lemma \ref{2.1} is applicable. This together with  \eqref{HL2} \hs{enables}
   us to finish the proof of   this part.
\end{proof}

\begin{proof}[Proof of \eqref{enumerate2} in Theorem \ref{the1.1}]
Along the lines of \cite{[JWY]}, we see that
\be\ba\label{3.44}
\f{d}{dt} \Psi(t) + \f\kappa2\int|\nabla \theta|^{2}dx+&\f{1}{2}\int \rho |\dot{u}|^{2}dx+\f{1}{2}\int_{{\mathbb{R}^3}\cap\{|u|>0\}}
|u|^2\big|\nabla u\big|^2\\
&\leq C\left(\int \hs{\rho }|\theta|^{3} dx +\int\rho |u|^{2}|\theta|^{2}dx\right),
 \ea\ee
 where
$$ \Psi(t)=\int\B[\f\mu2|\nabla u|^{2}+(\mu+\lambda)(\text{div\,}u  )^2+\f{1}{2\mu+\lambda}P^2 -2P\text{div\,}u +\f{C_3 c_\nu}{2}\rho\theta^2 +\f{C_4 +1}{2\mu}\rho| u|^4]dx.$$
To bound the right hand side of \eqref{3.44},
we infer from the H\"older inequality \eqref{hiL} or interpolation characteristic \eqref{Interpolation characteristic}
and Sobolev embedding \eqref{sl} in Lorentz spaces that
 \be\ba
 \label{2.11}\Big\|  |u|^{2}\Big\|_{L^{\f{2q}{q-1},2}\hs{(\mathbb{R}^{3})}}& \leq C \hs{\Big\|}  |u|^{2} \Big\|^{1-\f{3}{2q}}_{L^{2}(\mathbb{R}^{3})}\Big\|
 |u|^{2}\Big\|^{\hs{\f{3}{2q}}}_{L^{6,2}(\mathbb{R}^{3})}\\& \leq C\Big\| |u|^{2} \Big\|^{1-\f{3}{2q}}_{L^{2}(\mathbb{R}^{3})}\Big\|\nabla
|u|^{2}\Big\|^{\f{3}{2q}}_{L^{2}(\mathbb{R}^{3})}.
\ea\ee
Likewise,
 \be\ba
 \label{2.12} \| \rho^{\f12}\theta \|_{L^{\f{2q}{q-1},2}\hs{(\mathbb{R}^{3})}}  \leq C \| \theta  \|^{1-\f{3}{2q}}_{L^{2}(\mathbb{R}^{3})} \|\nabla
\theta \|^{\f{3}{2q}}_{L^{2}(\mathbb{R}^{3})}.
\ea\ee
The H\"older's inequality \eqref{hiL}, \eqref{2.11}, \eqref{2.12} and the Young inequality show  us that
\be\ba\label{3.42}
\int\rho |u|^{2}|\theta|^{2}dx&=\int\rho^{\f12}|\theta|\rho^{\f12}|u|^{2}|\theta| dx  \\
&\leq \|\theta\|_{L^{q,\infty }(\mathbb{R}^{3})}\|\rho^{\f12}\theta\|_{L^{\f{2q}{q-1},2}(\mathbb{R}^{3})}
\|\rho^{\f12}|u|^{2}\|_{L^{\f{2q}{q-1},2}(\mathbb{R}^{3})}\\
&\leq \|\theta\|_{L^{q,\infty}(\mathbb{R}^{3})}\|\rho^{\f12}\theta\|^{1-\f{3}{2q}}_{L^{2}(\mathbb{R}^{3})}\|
\rho^{\f12}\theta\|^{\f{3}{2q}}_{L^{6}(\mathbb{R}^{3})}\|\rho^{\f12}|u|^{2}\|_{L^{2}(\mathbb{R}^{3})}^{1-\f{3}{2q}}\|
\rho^{\f12}|u|^{2}\|^{\f{3}{2q}}_{L^{6}(\mathbb{R}^{3})}\\
&\leq  C \|\theta\|_{L^{q,\infty}(\mathbb{R}^{3})}\|\rho^{\f12}\theta\|^{1-\f{3}{2q}}_{L^{2}(\mathbb{R}^{3})}\|\rho^{\f12}|u|^{2}\|_{L^{2}(\mathbb{R}^{3})}^{1-\f{3}{2q}} (\|
 \nabla\theta\|^{\f{3}{q}}_{L^{2}(\mathbb{R}^{3})}+\|\nabla|u|^{2}\|^{\f{3}{q}}_{L^{2}(\mathbb{R}^{3})})
 \\
&\leq  C \|\theta\|^{\f{2q}{2q-3}}_{L^{q,\infty}(\mathbb{R}^{3})}\|\rho^{\f12}\theta\|_{L^{2}(\mathbb{R}^{3})}\|\rho^{\f12}|u|^{2}\|_{L^{2}(\mathbb{R}^{3})}
 +\eta_{1}
 \|\nabla\theta\|^{2}_{L^{2}(\mathbb{R}^{3})}+\eta_{2}\|\nabla|u|^{2}\|^{2}_{L^{2}(\mathbb{R}^{3})} \\
&\leq  C \|\theta\|^{\f{2q}{2q-3}}_{L^{q,\infty}(\mathbb{R}^{3})}\big(\|\rho^{\f12}\theta\|_{L^{2}(\mathbb{R}^{3})}^{2}+\|\rho^{\f12}|u|^{2}\|_{L^{2}(\mathbb{R}^{3})}^{2}\big)
 +\hs{\eta_{1}}
 \|\nabla\theta\|^{2}_{L^{2}(\mathbb{R}^{3})}+\eta_{2}\|\nabla|u|^{2}\|^{2}_{L^{2}(\mathbb{R}^{3})}.  \ea
\ee
Arguing \hs{in} the same manner as the above, we see that
\be\ba \label{3.43}
\int\rho |\theta|^{3}dx&=\int\rho^{\f12}|\theta|\rho^{\f12}\theta|\theta| dx  \\
&\leq \|\theta\|_{L^{q,\infty}(\mathbb{R}^{3})}\|\rho^{\f12}\theta\|_{L^{\f{2q}{q-1},2}(\mathbb{R}^{3})}^{2}\\
&\leq \|\theta\|_{L^{q,\infty}(\mathbb{R}^{3})}\|\rho^{\f12}\theta\|^{2-\f{3}{q}}_{L^{2}(\mathbb{R}^{3})}\|
\rho^{\f12}\theta\|^{\f{3}{q}}_{L^{6,2}( \mathbb{ R}^{3})}\\
&\leq  C \|\theta\|_{L^{q,\infty}(\mathbb{R}^{3})}\|\rho^{\f12}\theta\|^{2-\f{3}{q}}_{L^{2}(\mathbb{R}^{3})}\|
 \theta\|^{\f{3}{q}}_{L^{6,2}(\mathbb{R}^{3})}\\
&\leq  C \|\theta\|_{L^{q,\infty}(\mathbb{R}^{3})}\|\rho^{\f12}\theta\|^{2-\f{3}{q}}_{L^{2}(\mathbb{R}^{3})}\|
 \nabla\theta\|^{\f{3}{q}}_{L^{2}(\mathbb{R}^{3})}\\
&\leq  C \|\theta\|^{\f{2q}{2q-3}}_{L^{q,\infty}(\mathbb{R}^{3})}\|\rho^{\f12}\theta\|_{L^{2}(\mathbb{R}^{3})}^{2}
 +\hs{\eta_{3}}\| \nabla\theta\|^{2}_{L^{2}(\mathbb{R}^{3})}.
 \ea
\ee

 Substituting  \eqref{3.42} and \eqref{3.43} into \eqref{3.44}, we have
\be\ba\label{3.441}
\f{d}{dt}\Psi(t)+ \hs{\f\kappa4}\int|\nabla \theta|^{2}dx+&\f{1}{2}\int \rho |\dot{u}|^{2}dx+\f{1}{4}\int_{{\mathbb{R}^3}\cap\{|u|>0\}}
|u|^2\big|\nabla u\big|^2\\
&\leq C \|\theta\|^{\f{2q}{2q-3}}_{L^{q,\infty}(\mathbb{R}^{3})}\big(\|\rho^{\f12}\theta\|_{L^{2}(\mathbb{R}^{3})}^{2}+\|\rho^{\f12}|u|^{2}\|_{L^{2}(\mathbb{R}^{3})}^{2}\big)\\
& \leq C\|\theta\|_{L^{q,\infty}(\mathbb{R}^{3})}^{\f{2q}{2q-3}}\Psi(t),
 \ea\ee
  where we used the fact that
 $$\int (\mu+\lambda)(\text{div\,}u)^2+\f{1}{2\mu+\lambda}P^2-2P\text{div\,}u +\f{C_3C_\nu}{2}\rho\theta^2dx\geq \int \rho\theta^2dx,$$
 provided that the constant $C_3$ is suitable large enough.

Thanks to interpolation characteristic \eqref{Interpolation characteristic} or the H\"older inequality \eqref{hiL}, applying lemma  \ref{lemma2.2} with $a=2, b=6$ and \eqref{Lorentzincrease}, we see that
\be\label{2.20001}
 \|\theta\|^{p_{\tau}}_{L^{q_{\tau},\infty}(\mathbb{R}^{3})}\leq
 \|\theta\|^{p(1-\tau)}_{L^{q ,\infty}(\mathbb{R}^{3})} \|\theta\|^{c_{0}\tau}_{L^{6,\infty}\hs{(\mathbb{R}^{3})}}\leq C
 \|\theta\|^{p(1-\tau)}_{L^{q ,\infty}(\mathbb{R}^{3})} \|\theta\|^{c_{0}\tau}_{L^{6 }(\mathbb{R}^{3})}\leq C
 \|\theta\|^{p(1-\tau)}_{L^{q ,\infty}(\mathbb{R}^{3})} \|\nabla\theta \|^{c_{0}\tau}_{L^{2 }(\mathbb{R}^{3})},
\ee
where $c_{0}$ is \hs{a constant to be} determined  later.

Since the pair $(p_{\tau}, q_{\tau})$ also meets $2/p_{\tau}+3/q_{\tau}=2$, we insert \hs{\eqref{2.20001} into \eqref{3.441}} to obtain
\be\ba\label{3.443}
\f{d}{dt}\Psi(t)+ \f\kappa4\int|\nabla \theta|^{2}dx+&\f{1}{4}\int \rho |\dot{u}|^{2}dx+\f{1}{4}\int_{{\mathbb{R}^3}\cap\{|u|>0\}}
|u|^2\big|\nabla u\big|^2\\
&\leq C
 \|\theta\|^{p(1-\tau)}_{L^{q ,\infty}(\mathbb{R}^{3})} \|\nabla\theta \|^{c_{0}\tau}_{L^{2 }(\mathbb{R}^{3})}\Psi(t)\\
  &\leq C\| \theta \|^{\f{2p(1-\tau)}{2-\hs{c_{0}}\tau}}_{L^{q,\infty}(\mathbb{R}^{3})} \Psi(t)^{\f{2}{2-\hs{c_{0}}\tau}} +\f{\kappa}{8} \| \nabla \theta  \|^{2}_{L^{2 }(\mathbb{R}^{3})},
 \ea\ee
Before going further, we take  $$\delta=\f{(2-\hs{c_{0}})\tau}{2-\hs{c_{0}}\tau},~~\hs{c_{0}}=\f{4}{3},$$
in the last relation.
Therefore, dropping some positive terms in \hs{\eqref{3.443}}, we  \hs{have}
 $$\f{d}{dt}  \Psi(t)\leq C\| \theta \|^{p(1-\delta)}_{L^{q,\infty}(\mathbb{R}^{3})} \Psi(t)^{4(1+2\delta) }.$$
In view of  $\kappa\in[0,1]$, we know that $\delta\in[0,1]$.
Finally, the proof is now based on applications of
Lemma \ref{2.1} and  \eqref{HL2}.
\end{proof}

\subsection{ Continuation theorem in anisotropic Lebesgue spaces}
\begin{proof}[Proof of \hs{\emph{(1)} in} Theorem \ref{the1.2}]
The interpolation inequality
  \eqref{interi} in   anisotropic Lebesgue \hs{spaces} entails that
\be \ba\label{anLI4.1}\|\rho E \|_{L_{1}^{\f{2q_{1}}{q_{1}-2}}L_{2}^{\f{2q_{2}}{q_{2}-2}}
L_{3}^{\f{2\hs{q}_{3}}{\hs{q}_{3}-2}}(\mathbb{R}^{3})}
&\leq  C\|  \rho   E \|^{1-(\f{1}{q_{1}}+\f{1}{q_{2}}+\f{1}{q_{3}})}_{L^2(\mathbb{R}^{3})}
\|  \rho   E \|^{(\f{1}{q_{1}}+\f{1}{q_{2}}+\f{1}{q_{3}})}
_{L^{\overrightarrow{m}}(\mathbb{R}^{3})} \\&\leq  C\|  \rho   E \|^{1-(\f{1}{q_{1}}+\f{1}{q_{2}}+\f{1}{q_{3}})}_{L^2(\mathbb{R}^{3})}
\|      E \|^{(\f{1}{q_{1}}+\f{1}{q_{2}}+\f{1}{q_{3}})}
_{L^{\overrightarrow{m}}(\mathbb{R}^{3})},
\ea\ee
with $\overrightarrow{m}$ defined by
$\overrightarrow{m}=\B(\f{\f{1}{q_{1}}+\f{1}{q_{2}}+\f{1}{q_{3}}}
{\f12(\f{1}{q_{1}}+\f{1}{q_{2}}+\f{1}{q_{3}})-\f{1}{\hs{q}_{1}}},\f{\f{1}{q_{1}}+\f{1}{q_{2}}+\f{1}{q_{3}}}
{\f12(\f{1}{q_{1}}+\f{1}{q_{2}}+\f{1}{q_{3}})-\f{1}{\hs{q}_{2}}},\f{\f{1}{q_{1}}+\f{1}{q_{2}}+\f{1}{q_{3}}}
{\f12(\f{1}{q_{1}}+\f{1}{q_{2}}+\f{1}{q_{3}})-\f{1}{\hs{q}_{3}}}\B)$.

The Sobolev
embedding theorem \eqref{zc} in anisotropic Lebesgue \hs{spaces}
guarantees
\be \ba\label{anLI2}\|\rho E \|_{L_{1}^{\f{2q_{1}}{q_{1}-2}}L_{2}^{\f{2q_{2}}{q_{2}-2}}
L_{3}^{\f{2\hs{q}_{3}}{\hs{q}_{3}-2}}(\mathbb{R}^{3})}
&\leq  C\|    \rho E \|^{1-(\f{1}{q_{1}}+\f{1}{q_{2}}+\f{1}{q_{3}})}_{L^2(\mathbb{R}^{3})}
\|    \nabla E \|^{(\f{1}{q_{1}}+\f{1}{q_{2}}+\f{1}{q_{3}})}
_{L^{2}(\mathbb{R}^{3})}\hs{.} \ea\ee
It follows from  H\"older inequality \eqref{HIAL} in  anisotropic Lebesgue \hs{spaces} that
\be\ba\label{4.3}
 \int \rho  |u|^{2}|E|^{2}dx  &\leq C \||u|^{2}\|_{L^{\f{\overrightarrow{q} }{2}}(\mathbb{R}^{3})} \| \rho E^{2}\|_{L_{1}^{\f{\hs{q}_{1}}{\hs{q}_{1}-2}}L_{2}^{\f{\hs{q}_{2}}{\hs{q}_{2}-2}}
L_{3}^{\f{\hs{q}_{3}}{\hs{q}_{3}-2}}(\mathbb{R}^{3})}
\\
&\leq C \| u \|^2_{L^{ {\overrightarrow{q} } }(\mathbb{R}^{3})} \| \rho E \|^{2}_{L_{1}^{\f{2q_{1}}{q_{1}-2}}L_{2}^{\f{2q_{2}}{q_{2}-2}}
L_{3}^{\f{2\hs{q}_{3}}{\hs{q}_{3}-2}}(\mathbb{R}^{3})}\\
&\leq C \| u \|^2_{L^{ {\overrightarrow{q} } }(\mathbb{R}^{3})}
\|  \nabla  E \|^{2(\f{1}{q_{1}}+\f{1}{q_{2}}+\f{1}{q_{3}})}_{L^2(\mathbb{R}^{3})} \|      \rho E \|^{2-2(\f{1}{q_{1}}+\f{1}{q_{2}}+\f{1}{q_{3}})}_{L^2(\mathbb{R}^{3})}
\\
&\leq C\| u \|^{\f{2}{1-(\f{1}{q_{1}}+\f{1}{q_{2}}+\f{1}{q_{3}})}}_{L^{ {\overrightarrow{q} } }(\mathbb{R}^{3})}\| \rho E \|_{L^{2}(\mathbb{R}^{3})}^{2}
+\f{C_{2} }{8}\|  \nabla  E \|_{L^{2}(\mathbb{R}^{3})}^{2}.
 \ea\ee
By following exactly the lines of reasoning which led to \hs{\eqref{anLI4.1}}, we infer that
\be \ba\label{anLI4.3}\| \nabla u  \|_{L_{1}^{\f{2q_{1}}{q_{1}-2}}L_{2}^{\f{2q_{2}}{q_{2}-2}}
L_{3}^{\f{2\hs{q}_{3}}{\hs{q}_{3}-2}}(\mathbb{R}^{3})}
&\leq  C\|     \nabla u  \|^{1-(\f{1}{q_{1}}+\f{1}{q_{2}}+\f{1}{q_{3}})}_{L^2(\mathbb{R}^{3})}
\|   \nabla u  \|^{(\f{1}{q_{1}}+\f{1}{q_{2}}+\f{1}{q_{3}})}
_{L^{\overrightarrow{m}}(\mathbb{R}^{3})}.
 \ea\ee
Taking advantage of  \eqref{bsa}, we conclude that
$$\ba
\|\nabla u\|_{L^{\overrightarrow{m}}(\mathbb{R}^{3})}\leq& C (\|\hs{\textrm{div}}\, u\|_{L^{\overrightarrow{m}}(\mathbb{R}^{3})}+\|\hs{\textrm{curl}} \,u\|_{L^{\overrightarrow{m}}(\mathbb{R}^{3})})\\
\leq& C (\|G\|_{L^{\overrightarrow{m}}(\mathbb{R}^{3})}+\|P\|_{L^{\overrightarrow{m}}(\mathbb{R}^{3})}+\|\hs{\textrm{curl}} \,u\|_{L^{\overrightarrow{m}}(\mathbb{R}^{3})})\\
\leq& C (\|G\|_{L^{\overrightarrow{m}}(\mathbb{R}^{3})}+\|E\|_{L^{\overrightarrow{m}}(\mathbb{R}^{3})}+\|\hs{\textrm{curl}} \,u\|_{L^{\overrightarrow{m}}(\mathbb{R}^{3})})
\\
\leq& C (\|\nabla G\|_{L^{2}(\mathbb{R}^{3})}+\|\nabla E\|_{L^{2}(\mathbb{R}^{3})}+\|\nabla \hs{\textrm{curl}}\, u\|_{L^{2}(\mathbb{R}^{3})})
\\
\leq& C (\|\rho \dot{u}\|_{L^{2}(\mathbb{R}^{3})}+\|\nabla E\|_{L^{2}(\mathbb{R}^{3})})\hs{,}
\ea$$
where
the Sobolev
embedding theorem \eqref{zc} in anisotropic Lebesgue \hs{spaces}
was used again.
This means that
\be \ba\label{anLI}\| \nabla u  \|_{L_{1}^{\f{2q_{1}}{q_{1}-2}}L_{2}^{\f{2q_{2}}{q_{2}-2}}
L_{3}^{\f{2q_{3}}{q_{3}-2}}(\mathbb{R}^{3})}
&\leq  C\|     \nabla u  \|^{1-(\f{1}{q_{1}}+\f{1}{q_{2}}+\f{1}{q_{3}})}
_{L^2(\mathbb{R}^{3})}
(\|\rho \dot{u}\|_{L^{2}(\mathbb{R}^{3})}+\|\nabla E\|_{L^{2}(\mathbb{R}^{3})})^{(\f{1}{q_{1}}+\f{1}{q_{2}}+\f{1}{q_{3}})}. \ea\ee
The H\"older inequality in anisotropic Lebesgue \hs{spaces} and the latter relation yield
\be\ba\label{4.6}
\int\hs{\rho}|u|^{2}| \nabla u|^{2}dx&\leq C \|\rho|u|^{2}\|_{L^{\f{\overrightarrow{q} }{2}}(\mathbb{R}^{3})} \|| \nabla u|^{2}\|_{L_{1}^{\f{\hs{q}_{1}}{\hs{q}_{1}-2}}L_{2}^{\f{\hs{q}_{2}}{\hs{q}_{2}-2}}
L_{3}^{\f{\hs{q}_{3}}{\hs{q}_{3}-2}}(\mathbb{R}^{3})}\\
&\leq C \| |u|^{2}\|_{L^{\f{\overrightarrow{q} }{2}}(\mathbb{R}^{3})} \| \nabla u\|^{2}_{L_{1}^{\f{2q_{1}}{q_{1}-2}}L_{2}^{\f{2q_{2}}{q_{2}-2}}
L_{3}^{\f{2\hs{q}_{3}}{\hs{q}_{3}-2}}(\mathbb{R}^{3})}\\
&\leq C \|  u \|^2_{L^{ {\overrightarrow{q} } }(\mathbb{R}^{3})} \|     \nabla u  \|^{2-2(\f{1}{q_{1}}+\f{1}{q_{2}}+\f{1}{q_{3}})}
_{L^2(\mathbb{R}^{3})}
(\| \dot{u}\|_{L^{2}(\mathbb{R}^{3})}+\|\nabla E\|_{L^{2}(\mathbb{R}^{3})})^{2(\f{1}{q_{1}}+\f{1}{q_{2}}+\f{1}{q_{3}})}
\\&\leq
  C\|  u \|^{\f{2}{1-(\f{1}{q_{1}}+\f{1}{q_{2}}+\f{1}{q_{3}})}}
  _{L^{ {\overrightarrow{q} } }(\mathbb{R}^{3})}
  \| \nabla u \|_{L^{2}(\mathbb{R}^{3})}^{2}+
  \f{C_{2} }{8}\|  \nabla  E \|^{2}_{L^{2}(\mathbb{R}^{3})}+
  \f14\|  \rho \dot{u} \|^{2}_{L^{2}(\mathbb{R}^{3})}  .
  \ea\ee
Inserting \eqref{4.3} and \eqref{4.6} into \eqref{k221}, we observe that
\be\ba\label{k41}
&   \f{d}{dt}\hs{\Phi(t)}+ \f{1}{4}
\int \rho |\dot{u}|^{2}dx+ \f{C_{2} }{4}\int|\nabla E|^{2}dx\\
\leq& C\| u \|^{\f{2}{1-(\f{1}{q_{1}}+\f{1}{q_{2}}+\f{1}{q_{3}})}}
_{L^{ {\overrightarrow{q} } }(\mathbb{R}^{3})}(\|  \rho E\|^{2}_{L^{ 2} ( \mathbb{R}^{3})}+\|  \nabla u\|^{2}_{L^{ 2} (\mathbb{R}^{3})}).
\ea\ee
To finish the proof, we now merely have to combine the Gronwall's inequality
with \eqref{k41} and \eqref{HL2}.

\end{proof}

\begin{proof}[Proof of \hs{\emph{(2)} in} Theorem \ref{the1.2}]
Estimating the last term of (\ref{3.44}) \hs{by}
the
H\"older inequality \hs{in} anisotropic Lebesgue spaces, we see that
$$\ba
 \int \rho  |u|^{2}|\theta|^{2}dx  &\leq C \|\theta\|_{L^{\overrightarrow{q}}(\mathbb{R}^{3})} \| \sqrt{\rho} \theta \|_{L_{1}^{\f{2q_{1}}{q_{1}-1}}L_{2}^{\f{2\hs{q_{2}}}{q_{2}-1}}
L_{3}^{\f{2q_{3}}{q_{3}-1}}(\mathbb{R}^{3})}\| \sqrt{\rho} |u|^{2}\|_{L_{1}^{\f{2q_{1}}{q_{1}-1}}L_{2}^{\f{2\hs{q_{2}}}{q_{2}-1}}
L_{3}^{\f{2q_{3}}{q_{3}-1}}(\mathbb{R}^{3})}.
\ea$$
Due to \hs{the} interpolation inequality  \eqref{interi}, we get
$$\ba
\| \sqrt{\rho} \hs{\theta}\|_{L_{1}^{\f{2q_{1}}{q_{1}-1}}L_{2}^{\f{2\hs{q_{2}}}{q_{2}-1}}
L_{3}^{\f{2q_{3}}{q_{3}-1}}(\mathbb{R}^{3})}&\leq  C\| \sqrt{\rho} \theta \|^{1-\f12(\f{1}{q_{1}}+\f{1}{q_{2}}+\f{1}{q_{3}})}_{L^2(\mathbb{R}^{3})}
\|  \sqrt{\rho}   \theta \|^{\f12(\f{1}{q_{1}}+\f{1}{q_{2}}+\f{1}{q_{3}})}
_{L^{\overrightarrow{m}}(\mathbb{R}^{3})}\\&\leq  C\| \sqrt{\rho} \theta \|^{1-\f12(\f{1}{q_{1}}+\f{1}{q_{2}}+\f{1}{q_{3}})}_{L^2(\mathbb{R}^{3})}
\|    \nabla \theta \|^{\f12(\f{1}{q_{1}}+\f{1}{q_{2}}+\f{1}{q_{3}})}
_{L^{2}(\mathbb{R}^{3})}.
\ea$$
Likewise,
$$\ba
\| \sqrt{\rho} |u|^{2}\|_{L_{1}^{\f{2q_{1}}{q_{1}-1}}L_{2}^{\f{2\hs{q_{2}}}{q_{2}-1}}
L_{3}^{\f{2q_{3}}{q_{3}-1}}(\mathbb{R}^{3})}&\leq  C\| \sqrt{\rho} |u|^{2} \|^{1-\f12(\f{1}{q_{1}}+\f{1}{q_{2}}+\f{1}{q_{3}})}_{L^2(\mathbb{R}^{3})}
\|    \nabla |u|^{2} \|^{\f12(\f{1}{q_{1}}+\f{1}{q_{2}}+\f{1}{q_{3}})}
_{L^{2}(\mathbb{R}^{3})}.
\ea$$
It is easy to see \hs{from the} Young inequality that
$$\ba
 &\int \rho  |u|^{2}|\theta|^{2}dx \\\leq& C
 \|\theta\|_{L^{ {\overrightarrow{q} } }(\mathbb{R}^{3})} \hs{\|} \sqrt{\rho} \theta \|
 ^{1-\f12(\f{1}{q_{1}}+\f{1}{q_{2}}+\f{1}{q_{3}})}_{L^2(\mathbb{R}^{3})}
 \| \sqrt{\rho} |u|^{2} \|^{1-\f12(\f{1}{q_{1}}+\f{1}{q_{2}}+\f{1}{q_{3}})}_{L^2(\mathbb{R}^{3})}
(\|    \nabla \theta \|^{(\f{1}{q_{1}}+\f{1}{q_{2}}+\f{1}{q_{3}})}
_{L^{2}(\mathbb{R}^{3})}+\|    \nabla |u|^{2} \|^{(\f{1}{q_{1}}+\f{1}{q_{2}}+\f{1}{q_{3}})}
_{L^{2}(\mathbb{R}^{3})})\\
\leq& C\|\theta \|^{\f{2}{2-(\f{1}{q_{1}}+\f{1}{q_{2}}+\f{1}{q_{3}})}}
_{L^{ {\overrightarrow{q} } }(\mathbb{R}^{3})}(\| \sqrt{\rho}\theta \|_{L^{2}(\mathbb{R}^{3})}^{2} +\| \sqrt{\rho}|u|^{2} \|_{L^{2}(\mathbb{R}^{3})}^{2})+\f{1}{16}\hs{(\|  \nabla \theta \|_{L^{2}(\mathbb{R}^{3})}^{2}+\|  \nabla |u|^{2} \|_{L^{2}(\mathbb{R}^{3})}^{2})}.
\ea$$
Similarly, we conclude that
$$\ba
 \int \rho   |\theta|^{3}dx  &\leq C \|\theta\|_{L^{ {\overrightarrow{q} } }(\mathbb{R}^{3})} \| \sqrt{\rho} \theta \|^{2}_{L_{1}^{\f{2q_{1}}{q_{1}-1}}L_{2}^{\f{2\hs{q_{2}}}{q_{2}-1}}
L_{3}^{\f{2q_{3}}{q_{3}-1}}(\mathbb{R}^{3})} \\&
\leq  C \|\theta\|_{L^{ {\overrightarrow{q} } }(\mathbb{R}^{3})}\| \sqrt{\rho} \theta \|^{2- (\f{1}{q_{1}}+\f{1}{q_{2}}+\f{1}{q_{3}})}_{L^2(\mathbb{R}^{3})}
\|    \nabla \theta \|^{ (\f{1}{q_{1}}+\f{1}{q_{2}}+\f{1}{q_{3}})}
_{L^{2}(\mathbb{R}^{3})}\\
 &\leq C\|\theta \|^{\f{2}{2-(\f{1}{q_{1}}+\f{1}{q_{2}}+\f{1}{q_{3}})}}_{L^{ {\overrightarrow{q} } }(\mathbb{R}^{3})} \| \sqrt{\rho}\theta \|_{L^{2}(\mathbb{R}^{3})}^{2} +
 \f{1}{16} \|  \nabla \theta \|^{2}_{L^{2}(\mathbb{R}^{3})}.
 \ea$$
As a consequence of \hs{the above} estimates and  \eqref{3.44}, we see that
$$\ba
\f{d}{dt}\Psi(t)+ \f\kappa4\int|\nabla \theta|^{2}dx+&\f{1}{2}\int \rho |\dot{u}|^{2}dx+\f{1}{4}\int_{{\mathbb{R}^3}\cap\{|u|>0\}}
|u|^2\big|\nabla u\big|^2\\&\leq
C\|\theta \|^{\f{2}{2-(\f{1}{q_{1}}+\f{1}{q_{2}}+\f{1}{q_{3}})}}
_{L^{ {\overrightarrow{q} } }(\mathbb{R}^{3})}(\| \sqrt{\rho}\theta \|_{L^{2}(\mathbb{R}^{3})}^{2}+\| \sqrt{\rho}|u|^{2} \|_{L^{2}(\mathbb{R}^{3})}^{2}).
 \ea$$
Combining this and Gronwall's lemma, we are now able to complete the proof.
\end{proof}

\section{Nonhomogeneous
incompressible Navier-Stokes equations}
We \hs{proceed with} the nonhomogeneous
incompressible Navier-Stokes equations similarly as \hs{in} the previous proof.
\begin{proof}[\wg{Proof of   Theorem \ref{the1.4}}]

Taking the $L^2$ inner product of the system with $u_t$ and \hs{integrating} by \hs{parts, we have}

$$\ba
\f{\mu}{2}\f{d}{dt}\int|\nabla u|^{2}dx+\int \rho |u_{t}|^{2}dx
=-\hs{\int}\rho u\cdot\nabla u\cdot u_{t}dx\hs{.}
\ea$$
From the Young inequality, we infer that
$$-\hs{\int}\rho u\cdot\nabla u\cdot u_{t}dx\leq \f{1}{4}\int \rho |u_{t}|^{2}dx+C\hs{\int}\rho|u|^{2}| \nabla u|^{2}dx\hs{.}$$
\hs{Next, we} write (\ref{NNS}) in the form
$$-\mu \Delta u+\nabla \Pi=-\rho u_{t}-\rho u\cdot\nabla u \hs{.}$$
Resorting to the classical Stokes estimate, we discover that
$$\|\nabla^{2}u\|^{2}_{L^{2}(\mathbb{R}^{3})}\leq C_{6}
(\|\rho u_{t}\|^{2}_{L^{2}(\mathbb{R}^{3})}
+\|\rho u\cdot\nabla u\|_{L^{2}(\mathbb{R}^{3})}^{2})\hs{,}$$
namely,
$$
\f{1}{4C_{6}}\|\nabla^{2}u\|^{2}_{L^{2}(\mathbb{R}^{3})}\leq \f14\hs{(\|\rho u_{t}\|_{L^{2}(\mathbb{R}^{3})}^{2}+\|\rho u\cdot\nabla u\|_{L^{2}(\mathbb{R}^{3})}^{2})}\hs{.}
$$
This in turn implies
\be\ba\label{nnskey}
\f12\int \rho |u_{t}|^{2}dx
+\f{\mu}{2}\f{d}{dt}\int|\nabla u|^{2}dx +\f{1}{4C_{6}}\int|\nabla^{2} u|^{2}dx\leq C\hs{\int}\rho|u|^{2}| \nabla u|^{2}dx.
\ea\ee

\hs{(1)} \hs{Repeating} the \hs{process of deduction in} \eqref{i11}, we obtain
$$\ba
\hs{\int}\rho|u|^{2}| \nabla u|^{2}dx&\leq \||u|^{2}\|_{L^{\f{q}{2},\infty}(\mathbb{R}^{3})}\|    \hs{|\nabla u|^{2}}\|_{L^{\f{q}{q-2},1} (\mathbb{R}^{3})}\\&\leq
\|u\|^{2}_{L^{\hs{q},\infty}(\mathbb{R}^{3})}\|   \nabla u\|^{2}_{L^{\f{2q}{q-2},2} (\mathbb{R}^{3})} \\
&\leq C\|  u\|^{2}_{L^{q,\infty}(\mathbb{R}^{3})}\| \nabla^2 u \|^{\f6q}_{L^{2}(\mathbb{R}^{3})}\| \nabla u\|^{2-\f{6}{q}}_{L^{ 2} (\mathbb{R}^{3})}\\
&\leq C\| u\|^{\f{2q}{q-3}}_{L^{q,\infty}(\mathbb{R}^{3})}\|  \nabla u\|^{2}_{L^{ 2} (\mathbb{R}^{3})}+\f{1}{32}\| \nabla^2  u \|^{2}_{L^{2}(\mathbb{R}^{3})}\hs{.}
\ea$$
\hs{Lemma} \ref{lemma2.2} with $a=1, b=6$ and $c_{0}=4$ gives
$$\ba
\| u\|^{p_{\tau}}_{L^{q_{\tau},\infty}(\mathbb{R}^{3})}&\leq
\| u\|^{p(1-\tau)}_{L^{q ,\infty}(\mathbb{R}^{3})} \| u\|^{4\tau}_{L^{6,\infty}\hs{(\mathbb{R}^{3})}}\\&\leq \hs{C}
 \| u\|^{p(1-\tau)}_{L^{q ,\infty}(\mathbb{R}^{3})} \| u\|^{4\tau}_{L^{6,2 }(\mathbb{R}^{3})}\\&\leq \hs{C}
 \| u\|^{p(1-\tau)}_{L^{q ,\infty}(\mathbb{R}^{3})} \|\nabla u \|^{4\tau}_{L^{2 }(\mathbb{R}^{3})}\hs{.}\ea$$
Plugging these \hs{estimates} into \eqref{nnskey}, we find
\be\label{5.3}
\f12\int \rho |u_{t}|^{2}dx
+\f{\mu}{2}\f{d}{dt}\int|\nabla u|^{2}dx +\f{1}{8C_{6}}\int|\nabla^{2} u|^{2}dx\leq C
 \| u\|^{p(1-\tau)}_{L^{q ,\infty}(\mathbb{R}^{3})} \|\nabla u \|^{\hs{4\tau+2}}_{L^{2 }(\mathbb{R}^{3})}.
 \ee
Applying Lemma \ref{2.1} to \eqref{5.3} and using \eqref{kim},
 we complete the proof.

\hs{(2)} Exactly as in the derivation of \eqref{4.3}, we \hs{have}
$$\ba
\hs{\int}\rho|u|^{2}| \nabla u|^{2}dx
&\leq C \||u|^{2}\|_{L^{\f{\overrightarrow{q} }{2}}(\mathbb{R}^{3})} \|| \nabla u|^{2}\|_{\hs{L_{1}^{\f{q_{1}}{q_{1}-2}}L_{2}^{\f{q_{2}}{q_{2}-2}}
L_{3}^{\f{q_{3}}{q_{3}-2}}}(\mathbb{R}^{3})}\\
&\leq C \| u \|^2_{L^{ {\overrightarrow{q} } }(\mathbb{R}^{3})} \|  \nabla u \|^{2}_{L_{1}^{\f{2q_{1}}{q_{1}-2}}L_{2}^{\f{2q_{2}}{q_{2}-2}}
L_{3}^{\f{2\hs{q_{3}}}{\hs{q_{3}}-2}}(\mathbb{R}^{3})}\\
&\leq C \| u \|^2_{L^{ {\overrightarrow{q} } }(\mathbb{R}^{3})}
\|  \nabla^{2} u \|^{2(\f{1}{q_{1}}+\f{1}{q_{2}}+\f{1}{q_{3}})}_{L^2(\mathbb{R}^{3})} \|  \nabla  u \|^{\hs{2-2(\f{1}{q_{1}}+\f{1}{q_{2}}+\f{1}{q_{3}})}}_{L^2(\mathbb{R}^{3})}
\\
&\leq C\| u \|^{\f{2}{1-(\f{1}{q_{1}}+\f{1}{q_{2}}+\f{1}{q_{3}})}}_{L^{ {\overrightarrow{q} } }(\mathbb{R}^{3})}\|  \nabla  u \|_{L^{2}(\mathbb{R}^{3})}^{2}+\f{1}{8C_{6}}\|  \hs{\nabla^{2} u} \|_{L^{2}(\mathbb{R}^{3})}^{2}\hs{.}
\ea$$
As a consequence, we get
\be\ba\label{5.2}
\f12\int \rho |u_{t}|^{2}dx
+\f{\mu}{2}\f{d}{dt}\int|\nabla u|^{2}dx +\f{1}{8C_{6}}\int|\nabla^{2} u|^{2}dx\leq C \| u \|^{\f{2}{1-(\f{1}{q_{1}}+\f{1}{q_{2}}+\f{1}{q_{3}})}}_{L^{ {\overrightarrow{q} } }(\mathbb{R}^{3})}\|  \nabla  u \|_{L^{2}(\mathbb{R}^{3})}^{2}.
\ea\ee

\hs{(3)} In view of the H\"older inequality and Sobolev inequality, we get
$$\ba
\hs{\int}\rho|u|^{2}| \nabla u|^{2}dx
&\leq \||u|^{2}\|_{L^{3}(\mathbb{R}^{3})}\||\nabla u|^{2}\|_{L^{\f32}(\mathbb{R}^{3})}\\
&\leq \|u \|^{2}_{L^{6}(\mathbb{R}^{3})}\|\hs{\nabla u}\|^{2}_{L^{3}(\mathbb{R}^{3})}
\\
&\leq C\| \nabla u \|^{2}_{L^{3}(\mathbb{R}^{3})}\|\nabla u \|^{2}_{L^{2}(\mathbb{R}^{3})}\hs{.}\ea$$
This leads to
\be\label{5.4}
\f12\int \rho |u_{t}|^{2}dx
+\f{\mu}{2}\f{d}{dt}\int|\nabla u|^{2}dx +\f{1}{8C_{6}}\int|\nabla^{2} u|^{2}dx\leq C \| \nabla u \|^{2}_{L^{3}(\mathbb{R}^{3})}
\|\nabla u \|^{2}_{L^{2}(\mathbb{R}^{3})}.
\ee
With \eqref{nnskey},  \eqref{5.2}  and \eqref{5.4} in hand, employing classical Gronwall's lemma, we complete  the proof by \eqref{kim}.

 \end{proof}

\section*{Acknowledgement}
Wang was partially supported by  the National Natural
Science Foundation of China under grant (No. 11971446  and No. 11601492). Wei was partially supported by the National Natural Science Foundation of China under grant (No. 11601423, No. 11701450, No. 11701451, No. 11771352, No. 11871057) and Scientific Research Program Funded by Shaanxi Provincial Education Department (Program No. 18JK0763).
 The research of Wu was partially supported by the National
Natural Science Foundation of China under grant No. 11771423 and No. 11671378. Ye is supported by the NSFC (No. 11701145).

\end{document}